\documentclass[reqno]{article}
\usepackage{amsmath,amsfonts,amssymb,bbm,epsfig}
\usepackage{amsthm}

\newtheorem{theorem}{Theorem}[section]
\newtheorem{lemma}[theorem]{Lemma}

\newtheorem{corollary}[theorem]{Corollary}

\theoremstyle{definition}
\newtheorem{remark}[theorem]{Remark}

\renewcommand{\le}{\leqslant}
\renewcommand{\ge}{\geqslant}

\newcommand\ZZ{{\mathbb Z}}
\newcommand\RR{{\mathbb R}}
\newcommand\Bi{{\mathrm{Bin}}}
\newcommand\Po{{\mathrm{Po}}}

\newcommand\la{\lambda}

\renewcommand\Pr{{\mathbb P}}

\newcommand\E{{\mathbb E}}
\newcommand\Var{{\mathrm{Var}}}
\newcommand\Covar{{\mathrm{Cov}}}

\newcommand\cc{{\mathrm{c}}}

\newcommand\cG{\mathcal{G}}

\newcommand\cS{\mathcal{S}}

\newcommand\bb[1]{\bigl(#1\bigr)}

\newcommand\ceil[1]{\lceil#1\rceil}

\newcommand\Hrnp{H^r_{n,p}}

\newcommand\Hrsm{H^r_{s,m}}

\newcommand\Hrnpone{H^r_{n,p_1}}

\newcommand\Hrnpi{H^r_{n,p_i}}

\newcommand\ind[1]{\mathbbm{1}_{#1}}

\newcommand\bp{\mathfrak{X}}
\newcommand\Ei{E_\mathrm{int}}
\newcommand\Ep{E_\mathrm{per}}
\newcommand\tM{\tilde{M}}
\newcommand\tmu{\tilde{\mu}}
\newcommand\tsigma{\tilde{\sigma}}
\newcommand\tn{\tilde{n}}
\newcommand\bd{\bar{d}}

\newcommand\MM{S}

\begin{document}
\title{Counting dense connected hypergraphs via the probabilistic method}
\author{B\'ela Bollob\'as%
\thanks{Department of Pure Mathematics and Mathematical Statistics,
Wilberforce Road, Cambridge CB3 0WB, UK,
Department of Mathematical Sciences, University of Memphis, Memphis TN 38152, USA, and
London Institute for Mathematical Sciences, 35a South St., Mayfair, London W1K 2XF, UK.
E-mail: {\tt b.bollobas@dpmms.cam.ac.uk}.}
\thanks{Research supported in part by NSF grant DMS-1301614 and
EU MULTIPLEX grant 317532.}
\ and Oliver Riordan%
\thanks{Mathematical Institute, University of Oxford, Radcliffe Observatory Quarter, Woodstock Road, Oxford OX2 6GG, UK.
E-mail: {\tt riordan@maths.ox.ac.uk}.}}
\date{November 15, 2015; revised June 21, 2017}
\maketitle

\begin{abstract}
In 1990 Bender, Canfield and McKay gave an asymptotic formula for the number of
connected graphs on $[n]=\{1,2,\ldots,n\}$ with $m$ edges, whenever $n\to\infty$
and $n-1\le m=m(n)\le \binom{n}{2}$.
We give an asymptotic formula for the number $C_r(n,m)$ of connected $r$-uniform hypergraphs on $[n]$
with $m$ edges, whenever $r\ge 3$ is fixed
and $m=m(n)$ with $m/n\to\infty$, i.e., the average degree
tends to infinity. This complements recent results of 
Behrisch, Coja-Oghlan and Kang (the case $m=n/(r-1)+\Theta(n)$) and the present authors
(the case $m=n/(r-1)+o(n)$, i.e., `nullity' or `excess' $o(n)$). The proof is based on probabilistic methods, 
and in particular on a bivariate local limit theorem for the number of vertices and edges
in the largest component of a certain random hypergraph. The arguments are
much simpler than in the sparse case; in particular, we can use `smoothing' techniques
to directly prove the local limit theorem, without needing to first prove a central limit
theorem.
\end{abstract}


\section{Introduction and results}

Our aim in this paper is to prove a result that can be viewed in two equivalent ways:
as an asymptotic formula for the number of dense connected $r$-uniform hypergraphs with a
given number of vertices and edges, and as a local limit theorem concerning the numbers
of vertices and edges in the largest component of a certain random hypergraph.
This paper is a companion to~\cite{smoothing}, where we used related (but much more complicated)
methods to study the sparse case. Here we shall phrase our results in terms of the number
of vertices and the number of edges, rather than considering the nullity as in~\cite{smoothing}.
(The latter is a more natural parameter when it grows slowly, but not here.)

Throughout the paper we consider $r$-uniform hypergraphs, where $r\ge 2$ is fixed;
much of the time $r\ge 3$. A hypergraph
is \emph{connected} if it cannot be written as the vertex disjoint union of two
strictly smaller hypergraphs. (This is not the only possible sense of connectedness
when $r\ge 3$, but it is the most important one, and the only one we consider here.)
A basic problem in enumerative combinatorics is to count the number of `irreducible'
objects of a certain type according to certain size parameters. Here we
study $C_r(s,m)$, the number of connected $r$-uniform hypergraphs
on $[s]=\{1,2,\ldots,s\}$ with precisely $m$ edges.
We write $s$ rather than $n$ for the number of vertices in part for notational consistency
with~\cite{smoothing}, but also because in the bulk of the paper $n\ne s$ will be the number
of vertices in a certain random hypergraph; see Section~\ref{sec_prob}.

An asymptotic formula for $C_2(s,m)$ (the graph case) was proved by 
Bender, Canfield and McKay~\cite{BCMcK} in 1990, throughout the range $s-1\le m\le \binom{s}{2}$.
For $r\ge 3$, in 1997 Karo\'nski and \L uczak~\cite{KL_sparse} proved a result covering
the case $m=s/(r-1)+o(\log s/\log\log s)$:
this result concerns hypergraphs that are very close to trees.
This was generalized to $m=s/(r-1)+o(s^{1/3})$ in an extended abstract of Andriamampianina and Ravelomanana~\cite{AR}
in 2006. Recently, in~\cite{smoothing}, we proved a result covering the entire `sparse case' $m=s/(r-1)+o(s)$.
A formula covering the `middle range' $m=s/(r-1)+\Theta(s)$ 
was given by Behrisch, Coja-Oghlan and Kang~\cite{BC-OK2a,BC-OK2b}.
Our main result here covers the entire remaining range, the `dense case' $m/s\to\infty$. As usual
in this context, the statement involves an implicit definition, and so requires a little preparation.

For $\xi\in (0,1)$ define
\begin{equation}\label{phidef1}
  \Phi_r(\xi) = \frac{ \log(1/\xi) (1-\xi^r)}{(1-\xi^{r-1})(1-\xi)}.
\end{equation}
It is easy to check that, with $r\ge 2$ fixed, $\Phi_r$ is strictly decreasing, since each of the ratios
$\log(1/\xi)/(1-\xi)$ and $(1-\xi^r)/(1-\xi^{r-1})$ is.
Moreover, $\Phi_r(\xi)\to r/(r-1)$ as $\xi\to 1$ and $\Phi_r(\xi)\sim\log(1/\xi)\to\infty$ as $\xi\to 0$.
It is this latter limit which will be important here.
Since $\Phi_r$ is continuous, it defines a bijection from $(0,1)$ to $(r/(r-1),\infty)$.

Given $\bd > r/(r-1)$ let
\begin{equation}\label{xidef}
 \xi=\xi(\bd)=\Phi_r^{-1}(\bd),
\end{equation}
and set
\begin{equation}\label{Fdef}
 F_r(\bd) = \bd \log(1-\xi) - \frac{\bd}{r}\log(1-\xi^r) -\frac{\xi}{1-\xi}\log \xi -\log(1-\xi).
\end{equation}

\begin{theorem}\label{thenum}
Let $r\ge 2$ be fixed, and let $m=m(s)$ satisfy $m/s\to\infty$ as $s\to\infty$.
Let $\bd=rm/s$ be the average degree of an $m$-edge $r$-uniform hypergraph on $[s]$, and
let $\Hrsm$ be such a hypergraph chosen uniformly at random. Then
the probability $P_r(s,m)$ that $\Hrsm$ is connected
satisfies
\begin{equation}\label{Pform}
 P_r(s,m)  \sim \exp(-sF_r(\bd))
\end{equation}
as $s\to\infty$.

Furthermore, if in addition $m=o(s^{4/3})$, then
the number $C_r(s,m)$ of connected $m$-edge $r$-uniform hypergraphs on $[s]$ satisfies
\begin{equation}\label{Cform}
 C_r(s,m) \sim e^{-(r-1)\bd/2-\ind{r=2} \bd^2/4}\frac{s^{rm}}{m!r!^m} \exp(-sF_r(\bd))
\end{equation}
as $s\to\infty$, where $\ind{A}$ is the indicator function of $A$.
\end{theorem}

Writing $N=\binom{s}{r}$, we have $P_r(s,m)=C_r(s,m)/\binom{N}{m}$, so the formulae \eqref{Pform}
and \eqref{Cform} are equivalent up to a straightforward calculation; see Lemma~\ref{lPC}.
(In fact, for $r\ge 3$, \eqref{Cform} applies for $m=o(s^{3/2})$, not just $m=o(s^{4/3})$.)

\begin{remark}\label{rem2}
We focus on the case $r\ge 3$, since the case $r=2$ is covered by the result of
Bender, Canfield and McKay~\cite{BCMcK}. For our proof strategy, there is very 
little difference between the two situations; some formulae have extra terms when $r=2$,
since then certain error terms involving factors of $n^{-(r-1)}$ are not
totally negligible. When convenient, we assume $r\ge 3$, commenting briefly
on these extra terms.
\end{remark}

\begin{remark}
As we shall show later (in Lemma~\ref{asy}), for $r\ge 3$ we have
\[
 \xi = e^{-\bd} + \bd e^{-2\bd} +  O(\bd^2e^{-3\bd})
\]
and
\begin{equation}\label{Frd}
 F_r(\bd) = e^{-\bd} + \frac{\bd+1}{2}e^{-2\bd} + O(\bd^2 e^{-3\bd})
\end{equation}
as $\bd\to\infty$. (For $r=2$
the first term in the formulae above is the same as for $r\ge 3$, but the second
is different. The next term in the expansion~\eqref{Frd} is different in the cases
$r=2$, $r=3$ and $r\ge 4$.)
The probability that a vertex of $\Hrsm$ is isolated is
very close to $e^{-\bd}$, so the expected number of isolated vertices is approximately
$\mu=s e^{-\bd}$, and the Poisson intuition suggests that the probability that $\Hrsm$ has
no such vertex should be approximately $\exp(-\mu)$, at least when $\mu$ is bounded
or tends to infinity fairly slowly.
In turn, in this range we expect the presence of an isolated vertex to be the
main obstruction to connectivity. In the light of \eqref{Frd}, Theorem~\ref{thenum} says
that when $s\bd e^{-2\bd}=o(1)$ (corresponding to $\mu(s)=o(\sqrt{s/\log s})$), we have
\[
 P_r(s,m) \sim \exp(-se^{-\bd}).
\]
In other words, the intuition just described gives the right asymptotic answer in this (surprisingly large)
range. As a trivial special case, in the `very dense' case where $\bd-\log s\to\infty$, we
have $P_r(s,m)\sim 1$ and so $C_r(s,m)\sim \binom{N}{m}$ where $N=\binom{s}{r}$.
\end{remark}

\subsection{Comparison to related results}

The formulae appearing in Theorem~\ref{thenum} (the `dense case' $m/s\to\infty$), are superficially rather
different from those in Theorem 1.1 of~\cite{smoothing} (the `sparse case' $m=s/(r-1)+o(s)$)
and in (the corrected version of) Theorem 1.1 of Behrisch, Coja-Oghlan and Kang~\cite{BC-OK2b},
covering the `middle range' $m=s/(r-1)+\Theta(s)$. However, after a suitable change of notation,
they are actually rather similar, despite the different ranges of applicability.

Indeed, writing $\rho=1-\xi$, the definition of $\xi$ and hence of $\rho$ given by \eqref{xidef} is easily
seen to coincide with that in \cite{smoothing}. There, we set $\Psi_r(\rho)=(t-1)/s$
where $t=(r-1)m-s+1$ is the nullity, and $\Psi_r$ was given by a certain formula, (1.2) in~\cite{smoothing}.
Since $(t-1)/s=(r-1)m/s-1=(r-1)\bd/r-1$ in our present notation, and $\Psi_r(1-\xi)=(r-1)\Phi_r(\xi)/r-1$,
the quantities $\xi$ and $\rho$ here and in \cite{smoothing} are defined from the average degree $\bd$
in exactly the same way. Here we work mostly with $\xi$ rather than $\rho$ since in the dense
case $\xi\to 0$, which makes the asymptotics more intuitive. (In the sparse case $\rho\to 0$ instead.)

With $\xi=\Phi_r^{-1}(\bd)$ as above, since $\bd=rm/s$ we may write
\[
 \exp(-s F_r(\bd)) =
 \bb{ (1-\xi)^{-r}(1-\xi^r) }^m
  \bb{\xi^{\xi/(1-\xi)} (1-\xi)}^s.
\]
Equivalently, setting $\rho=1-\xi$,
\begin{equation}\label{Frho}
 \exp(-s F_r(\bd)) =
 \left(\frac{1- (1-\rho)^r}{\rho^r} \right)^m
  \bb{(1-\rho)^{(1-\rho)/\rho} \rho}^s .
\end{equation}
Thus we see that Theorem 1.1 of~\cite{smoothing} says exactly that
in the sparse case
\begin{equation}\label{Psparse}
 P_r(s,m) \sim e^{r/2+\ind{r=2}} \sqrt{\frac{3(r-1)}{2}} \exp(-s F_r(\bd) ).
\end{equation}

Turning to the middle range, as noted in the appendix to~\cite{smoothing}, the quantity
called $\Phi_d(r,\zeta)$ in~\cite{BC-OK2b} is exactly the right-hand side of \eqref{Frho}
above, i.e., simply $\exp(-s F_r(\bd))$ in our present notation.
Behrisch, Coja-Oghlan and Kang~\cite{BC-OK2b} write $\zeta$ for $\bd=rm/s$,
$d$ for the uniformity ($r$ here) and $r$ for what we call $\xi$. In our notation,
their main result says that in the middle range we have
\begin{equation}\label{universal1}
 P_r(s,m) \sim G_r(\bd) \exp(- s F_r(\bd) ),
\end{equation}
with
\begin{equation}\label{Gdef}
 G_r(\bd) = \frac{a_r(\bd)}{\sqrt{b_r(\bd)}} e^{g_r(\bd)},
\end{equation}
where, for $r\ge 3$,
\begin{eqnarray*}
 a_r(\bd) &=& 1-\xi^r -(1-\xi)(r-1)\bd \xi^{r-1}, \\
 b_r(\bd) &=& \bb{1-\xi^r +\bd(r-1)(\xi-\xi^{r-1})} (1-\xi^r) - r\bd \xi(1-\xi^{r-1})^2, \\
 g_r(\bd) &=& \frac{ (r-1) \bd (\xi-2\xi^r+\xi^{r-1}) } { 2(1-\xi^r)},
\end{eqnarray*}\
and
\begin{eqnarray*}
 a_2(\bd) &=& 1+\xi-\bd\xi, \\
 b_2(\bd) &=& (1+\xi)^2-2\bd\xi, \hbox{\quad and} \\
 g_2(\bd) &=& \frac{ 2\bd\xi+\bd^2\xi }{2(1+\xi)}.
\end{eqnarray*}

Since $\xi\sim e^{-\bd}$ as $\bd\to\infty$, it is trivial to check that as $\bd\to\infty$ we have
$a_r(\bd)\to 1$, $b_r(\bd)\to 1$ and $g_r(\bd)\to 0$. Hence
\begin{equation}\label{Grinf}
 G_r(\bd)\to 1 \hbox{\quad as\quad} \bd\to\infty,
\end{equation}
and \eqref{universal1} coincides
with our much simpler formula $\exp(-s F_r(\bd))$ in this case.

Finally, we noted in the appendix to~\cite{smoothing} (see equations (A.5) and (A.6))
that as $\bd\to r/(r-1)$ (corresponding to the nullity $t$ satisfying $t=o(s)$)
then $\xi\to 1$ and $G_r(\bd)\to e^{r/2+\ind{r=2}}\sqrt{\frac{3(r-1)}{2}}$, corresponding
to the pre-factor in \eqref{Psparse}. Collecting together these results,
we have the following
extension of the Bender--Canfield--McKay formula~\cite{BCMcK} to hypergraphs.

\begin{theorem}\label{thuniv}
Let $r\ge 2$ be fixed, and let $m=m(s)$ satisfy $m-s/(r-1)\to\infty$ and $m\le \binom{s}{r}$.
Then the proportion $P_r(s,m)$ of $m$-edge $r$-uniform hypergraphs on $[s]$ that
are connected satisfies
\begin{equation}\label{universal2}
 P_r(s,m) \sim G_r(\bd) \exp(-s F_r(\bd))
\end{equation}
where $\bd=rm/s$ is the average degree, and $F_r(\bd)$ and $G_r(\bd)$
are defined in \eqref{Fdef} and \eqref{Gdef}, with $\xi=\xi(r,\bd)$ defined in \eqref{xidef}.
\end{theorem}
\begin{proof}
Passing to a subsequence we may assume that the average degree $\bd$ satisfies
one of the conditions (i) $\bd\to r/(r-1)$, (ii) $\bd\to c$ with $r/(r-1)<c<\infty$, or (iii)
$\bd\to\infty$. Then we apply Theorem 1.1 of~\cite{smoothing} (in the form \eqref{Psparse} above),
Theorem 1.1 of~\cite{BC-OK2b}, or Theorem~\ref{thenum} above, recalling~\eqref{Grinf}.
\end{proof}

In the Appendix, we show that the $r=2$ case of the formula \eqref{universal2} does indeed coincide
with the (very different looking) formula in~\cite{BCMcK}.
Note that while one \emph{can} use the same formula in all cases (sparse, middle and dense), in
the sparse and dense cases this does not make much sense in practice, since the
formula simplifies greatly in these cases. Note also that while we have shown that the Behrisch--Coja-Oghlan--Kang
formula applies in the sparse and dense ranges too, so far as we know their \emph{proof} does not adapt to these
cases. Indeed, the sparse case treated in~\cite{smoothing} seems to require more complicated arguments
despite the relative simplicity of the formula.

\subsection{Probabilistic reformulation}\label{sec_prob}

We shall prove Theorem~\ref{thenum} (which, despite the trivial use of probability in the statement,
is a purely enumerative result) by probabilistic methods. Indeed, as we shall see in Section~\ref{sec_enum},
up to a (rather lengthy) calculation, Theorem~\ref{thenum} is essentially equivalent to a
local limit theorem for the numbers of vertices and edges in the largest component
of a certain random hypergraph. This is a similar situation to that in~\cite{smoothing}.

Turning to the details, given $r\ge 2$, $n\ge 1$ and $0<p<1$, let $\Hrnp$ be the random
$r$-uniform hypergraph on $[n]$ in which each of the $\binom{n}{r}$ possible edges is included
with probability $p$, independently of the others.
From now until Section~\ref{sec_enum} (where we return to the enumerative viewpoint)
we fix $r\ge 2$ and consider a function $d=d(n)$ satisfying
\begin{equation}\label{dconds}
 d\to \infty\hbox{\quad and\quad} \log n-d \to\infty,
\end{equation}
as $n\to\infty$. We consider the random hypergraph $\Hrnp$ with
\begin{equation}\label{pdef}
 p = p(n) = d  \frac{(r-1)!}{n^{r-1}},
\end{equation}
noting that the expected degree of a vertex is $p\binom{n-1}{r-1} = d(1+O(1/n))$.
Writing
\[
 \mu_0 = n e^{-d},
\]
then $\mu_0$ is roughly the expected number of isolated vertices in $\Hrnp$; the significance
of the second condition in \eqref{dconds} is that it implies $\mu_0\to\infty$ as $n\to\infty$.

Given $d>1/(r-1)$, define $\xi=\xi(d) \in(0,1)$ by
\begin{equation}\label{xieqn}
 \xi = \exp(- d (1-\xi^{r-1})).
\end{equation}
Since $\log(1/\xi)/(1-\xi^{r-1})$ is strictly decreasing, this uniquely defines $\xi$.
In fact, $\xi$ is the extinction probability of a certain branching process naturally associated
to the neighbourhood exploration process in $\Hrnp$; see Section~\ref{sec_bp}. 
It is not hard to see that $\xi\to 0$ as $d\to\infty$.
Substituting this back into \eqref{xieqn}, it follows that $\xi=e^{-d+o(d)}$. From this
it follows that $d\xi^{r-1}=o(1)$ and hence, by \eqref{xieqn} again,
\begin{equation}\label{xid}
 \xi\sim e^{-d} \hbox{\quad as\quad} d\to\infty.
\end{equation}
When we come to relate the probabilistic and enumerative viewpoints,
the average degree parameters $d$ and $\bd$ will not be (quite) equal (and neither will
the numbers of vertices, $n$ and $s$). However, for $d$ and $\bd$ related as they will
be, $\xi(\bd)$ as defined in the previous section and $\xi(d)$ as defined above will coincide.

For further background on the phase transition in the component structure of $\Hrnp$ see,
for example, Section 2 of~\cite{smoothing}. Here, we are well above the critical edge density.
As is well known (and we shall show below), when $d\to\infty$ then with very high probability
$\Hrnp$ has a (necessarily unique) `giant' component containing almost all the vertices --
in fact, all but around $\mu_0$ vertices.

We say that a sequence $((X_n,Y_n))$ of random variables
taking values in $\ZZ^2$ satisfies a \emph{local limit theorem}
with parameters $(\mu_X(n),\mu_Y(n))$ and $(\sigma_X^2(n),\sigma_Y^2(n))$ if 
for any sequence $(x_n,y_n)\in \ZZ^2$ we have
\[
 \Pr( (X_n,Y_n) = (x_n,y_n) ) = 
 \frac{1}{2\pi \sigma_X\sigma_Y}
 \left( \exp\left(-\frac{(x_n-\mu_X)^2}{2\sigma_X^2}-\frac{(y_n-\mu_Y)^2}{2\sigma_Y^2}\right) + o(1) \right),
\]
where we have partially suppressed the dependence on $n$ in the notation.
Note that, considering `almost worst case' values of $x_n$ and $y_n$, the $o(1)$ term can
be taken to be uniform over all $(x_n,y_n)\in \ZZ^2$.

Let $L_1(H)$ and $M_1(H)$ denote the numbers of vertices
and edges in the largest component of a hypergraph $H$, where `largest' means
with the most vertices, and we break ties arbitrarily. Here, then, is our local limit
theorem for the number of vertices and edges in the giant component of $\Hrnp$.

\begin{theorem}\label{thLLT}
Let $r\ge 2$ be fixed, let $d=d(n)\to\infty$ with $\log n-d\to\infty$, and set
$p = p(n) = d \frac{(r-1)!}{n^{r-1}}$. Let $L_1=L_1(\Hrnp)$ and $M_1=M_1(\Hrnp)$. Then we have
\begin{equation}\label{ebds}
 \E[L_1] = (1-\xi) n +o(1) \hbox{\quad and\quad} \E[M_1] = \frac{d(1-\xi^r)}{r} n +O(d),
\end{equation}
where $\xi=\xi(d)$ is defined in \eqref{xieqn}, and
\begin{equation}\label{vbds}
 \Var[L_1] \sim \sigma_L^2 \hbox{\quad and\quad} \Var[M_1] \sim \sigma_M^2
\end{equation}
where
\begin{equation}\label{sLM}
 \sigma_L^2=\sigma_L^2(n) = n e^{-d} \hbox{\quad and\quad} \sigma_M^2 = \sigma_M^2(n) = \frac{dn}{r}.
\end{equation}
Furthermore, the pair $(L_1,M_1)$ satisfies a local limit theorem with parameters
$(\E[L_1], \E[M_1])$ and $(\sigma_L^2,\sigma_M^2)$.
\end{theorem}
\begin{remark}
In the light of \eqref{ebds} and \eqref{vbds}, 
it is easy to check that we may replace the parameters for the means in the local limit theorem
above by $((1-\xi)n,d(1-\xi^r)n/r)$; see Lemma~\ref{parshift} below.
The key point is that the error terms $o(1)$ and $O(d)$ in \eqref{ebds} are (much) smaller than
$\sigma_L$ and $\sigma_M$, respectively.
\end{remark}

Although we are not aware of such accurate estimates for $\E[L_1]$ and $\E[M_1]$ in the literature, 
these are relatively straightforward.
The main point of Theorem~\ref{thLLT} is the local limit theorem. Analogous results
for the sparse regime ($d=1/(r-1)+o(1)$) and the `middle range' $d=1/(r-1)+\Theta(1)$
were proved in~\cite{smoothing} and by Behrisch, Coja-Oghlan and Kang~\cite{BC-OK2a}.
Indeed, recalling that the `nullity' $N_1$ studied in~\cite{smoothing} is defined to
be $(r-1)M_1-L_1+1$, after a little manipulation it is not hard to check that the
quantities $\rho_{r,\la}$ and $\rho_{r,\la}^*$ appearing in \cite{smoothing} as
approximations to $\E[L_1]/n$ and $\E[M_1]/n$ correspond to $(1-\xi)$  and
$(r-1)d (1-\xi^r)/r-(1-\xi)$ here. In other words, the formulae for the means match up;
the asymptotics of the variances are different in the different ranges considered
here and in~\cite{smoothing}.

The rest of the paper is organized as follows. In Sections~\ref{sec_bp}--\ref{sec_means}
we prepare the ground for the proof of Theorem~\ref{thLLT}. These sections contain
lemmas concerning, respectively, a certain branching process, basic properties of $\Hrnp$, and the
mean and variance of $L_1$ and $M_1$. Then, in Section~\ref{sec_pf} (the heart of the paper),
we use `smoothing' arguments to prove Theorem~\ref{thLLT}. Finally, in Section~\ref{sec_enum}
we deduce Theorem~\ref{thenum} via a somewhat involved calculation. In the Appendix,
we compare the $r=2$ case of our enumerative formula with that of Bender, Canfield and McKay.

\section{Branching process preliminaries}\label{sec_bp}

Given an integer $r\ge 2$ and a real number $d>0$, let $\bp_{r,d}$
be the Galton--Watson branching process defined as follows.
Start in generation $0$ with one individual. Each individual in generation $t$ has a random number
of \emph{groups} of $r-1$ children, with the number of groups having a Poisson distribution $\Po(d)$.
These children make up generation $t+1$. The numbers of children of all individuals in generation $t$
are independent of each other and of the history.

The branching process $\bp_{r,d}$ can be naturally viewed as a (possibly infinite) $r$-uniform hypergraph,
with a vertex for each individual, and a hyperedge for each group of children, consisting of these children
together with their parent. This hypergraph is of course an \emph{$r$-tree}, by which we simply
mean an $r$-uniform hypergraph that is a tree. (Often, when there is no danger of confusion,
we simply write `tree'.)
We write $|\bp_{r,d}|$ and $e(\bp_{r,d})$ for the number of vertices and edges in this $r$-tree, noting that
\[
 |\bp_{r,d}| = 1 + (r-1)e(\bp_{r,d}) \le\infty.
\]

Let $\rho=\rho_{r,d}=\Pr(|\bp_{r,d}|=\infty)$ be the probability that the branching process $\bp_{r,d}$
\emph{survives} forever, and $\xi=1-\rho$ its \emph{extinction probability}. Elementary properties
of branching processes (see, e.g., Athreya and Ney~\cite{AN})
imply that $\xi$ is the smallest solution in $[0,1]$
to the equation \eqref{xieqn}.
Moreover, if the branching factor $\la=(r-1)d$ is strictly greater than $1$, then \eqref{xieqn} has a
unique solution in $[0,1)$, and in particular $\xi<1$; otherwise, $\xi=1$.
Indeed, $\xi^{r-1}$ is the probability that all $r-1$ children
in a given group lead to finite trees, so $1-\xi^{r-1}$ is the probability that a given group
\emph{survives}, i.e., has infinitely many descendants.
From thinning properties of Poisson distributions, the number of such
surviving groups of children of the root is Poisson with mean $d(1-\xi^{r-1})$, and $\xi$ is exactly
the probability that this number is $0$. Hence $\xi$ satisfies \eqref{xieqn};
it is not hard to see that $\xi$ is indeed the smallest solution to this equation.

For comparison with the results in \cite{smoothing}, in terms of $\rho=1-\xi$ we may write \eqref{xieqn} as
\begin{equation}
 1-\rho = \exp( - d (1-(1-\rho)^{r-1})),
\end{equation}
which, writing $\la$ for $(r-1)d$, matches the definition given by (2.1) and (2.2) in~\cite{smoothing}.
Here, where $d$ tends to infinity, we will have $\xi\to 0$,
so $\xi$ is a more natural parameter to work with than $\rho$. This contrasts with the
situation in \cite{smoothing}, where $\rho\to 0$.

\subsection{The dual process}\label{sec_dual}

In the branching process $\bp_{r,d}$,
the \emph{groups} of children of the root may be classified into two types as above; those that
survive, and those that do not. By thinning properties of Poisson distributions, the number of groups
that do not survive has a Poisson distribution with mean $d\xi^{r-1}$, and is independent
of the number that do survive. For $d>1/(r-1)$ (the supercritical case) define the
\emph{dual parameter} $d^*$ by 
\begin{equation}\label{dstardef}
 d^* = d\xi^{r-1}.
\end{equation}
Then it follows that the conditional distribution of $\bp_{r,d}$ given $|\bp_{r,d}|<\infty$ is exactly
the unconditional distribution of the \emph{dual process} $\bp_{r,d^*}$.
It is not hard to check that the dual parameter coincides with that defined in \cite{smoothing};
the key point is that \eqref{dstardef} and \eqref{xieqn} imply $d^* e^{-(r-1)d^*} = d e^{-(r-1)d}$.

\subsection{Point probabilities}\label{ss_pp}

For $0\le k<\infty$ define
\[
 \pi_k = \pi_{k,r,d} = \Pr(e(\bp_{r,d})=k).
\]
The next lemma is a standard calculation, specialized to the particular offspring distribution
we have here. In this lemma, and much of this and the next two sections,
we adopt the convention of writing $s=1+(r-1)k$ for the number of vertices of an $r$-tree with
$k$ edges; this will make the formulae concerning `small' tree components much more concise.
\begin{lemma}\label{pkbd}
For any $r\ge 2$, $d>0$ and $k\ge 0$ we have
\[
 \pi_k = \pi_{k,r,d} = \frac{1}{s} \Pr(\Po(ds)=k) = \frac{s^{k-1}d^k}{k!} e^{-ds},
\]
where $s=1+(r-1)k$.
\end{lemma}
\begin{proof}
Consider the following alternative way of generating a random rooted $r$-tree.
Let $(a_1,a_2,\ldots)$ be independent and identically distributed, with $a_i\sim\Po(d)$.
Start with the root in generation $0$. Construct $a_1$ groups of $r-1$ children
of the root. Then proceed through these children one-by-one, assigning each child $a_i$
groups of $r-1$ children of its own, for $i=2,3,\ldots$. Continue with the new individuals 
(if any) in generation 2, and so on. By the definition of $\bp_{r,d}$ this tree has
the same distribution as $\bp_{r,d}$.
This construction stops if and only if, for some $i$, the first $i$ individuals `explored'
have between them at most, and hence exactly, $i-1$ children, i.e., $\frac{i-1}{r-1}$ \emph{groups}
of children.
This can only happen for $i\equiv 1$ modulo $r-1$. For $k\ge 0$ and $s=1+(r-1)k$, let
$\cS_k$ be the set of sequences $(a_i)_{1\le i\le s}$ of non-negative integers with the properties

\smallskip
(i) $\sum_{i\le s} a_i=k$ and

\smallskip
(ii) for all $1\le j<s$, $\sum_{i\le j} (r-1)a_i \ge j$.

\smallskip\noindent
Note that condition (i) may be written as $\sum_{i\le s} (r-1)a_i=s-1$. From the construction
above, $e(\bp_{r,d})=k$ if and only if our random sequence $(a_i)$ starts with a sequence in $\cS_k$.
By Spitzer's Lemma,
the probability of this is exactly
\[
 s^{-1} \Pr\Bigl( \sum_{i\le s} a_i =k\Bigr).
\]
Indeed, given any sequence $(a_i)_{i\le s}$ with $\sum_{i\le s} a_i=k$, there is a unique `rotation'
giving a sequence that also satisfies (ii) above, and since the $a_i$ are i.i.d., a sequence and
its rotations are equiprobable.

Since $\sum_{i\le s}a_i$ has a Poisson distribution with mean $ds$, the result follows.
\end{proof}

We state a trivial consequence for later reference.
\begin{corollary}\label{decay}
Let $r\ge 2$ be fixed. There is a constant $d_0$ such that for all $d\ge d_0$ and $k\ge 0$ we have
\[
 \pi_{k,r,d}\le e^{-d(s+1)/2},
\]
where $s=1+(r-1)k$.
\end{corollary}
\begin{proof}
For $k\ge 1$ we have $s/k\le r$, say. Since $k!\ge (k/e)^k$ it follows that
\begin{multline*}
 \pi_{k,r,d} =
 \frac{s^{k-1}d^k}{k!} e^{-ds} = 
 \frac{e^{-d}}{sk!} \bb{sde^{-(r-1)d}}^k 
 \le  e^{-d} \left(\frac{esd}{k} e^{-(r-1)d}\right)^k \\
  \le e^{-d} (erd e^{-(r-1)d})^k 
 \le e^{-d} e^{-(r-1)dk/2} = e^{-d(s+1)/2},
\end{multline*}
choosing $d_0$ so that $d\ge d_0$ implies $erde^{-(r-1)d/2}\le 1$ for all $d\ge d_0$.
Of course, the final bound holds for $k=0$ (i.e., $s=1$), since $\pi_{0,r,d}=e^{-d}$.
\end{proof}

The key consequence of Corollary~\ref{decay} is that the values $\pi_k$ decrease rapidly to zero
starting from $\pi_0=e^{-d}$.

In the subcritical case, we have the following simple result.
\begin{lemma}\label{negmo}
For any $r\ge 2$ and $d\ge 0$ with $(r-1)d<1$ we have
\[
 \E[ |\bp_{r,d}|^{-1} ] = \sum_{k\ge 0} s^{-1}\pi_k = 1-(r-1)d/r,
\]
where, as usual, $s=1+(r-1)k$.
\end{lemma}
\begin{proof}
Doubtless there is a direct algebraic proof of this. We outline a different argument:
with $r$ and $d$ fixed, consider the subcritical random hypergraph $\Hrnp$, $p=d (r-1)!n^{-(r-1)}$.
It is easy to check that for each fixed $k\ge 0$, the expected number of $k$-edge
tree components is $s^{-1} \pi_k n+o(n)$. (Either directly calculate the expectation or,
using a standard coupling argument,
note that $\pi_kn$ approximates the expected number of vertices in components
that are $k$-edge trees. The factor $1/s$ arises from counting components rather than vertices.)
Since the hypergraph is subcritical, for any $K(n)$ tending
to infinity, the expected number of components with at least $K(n)$ vertices is $o(n)$, 
as is the expected number of non-tree components. Since $\sum_k s^{-1}\pi_k$ converges,
it follows that the expectation $\mu_n$ of the number of components
satisfies $\mu_n= n \sum_k s^{-1}\pi_k +o(n)$.
On the other hand, adding edges one-by-one, the expected number of edges forming cycles
is small, and when an edge does not form a cycle the number of components goes down by $r-1$.
There are $dn/r+o(n)$ edges, so $\mu_n = n(1-(r-1)d/r)+o(n)$. Combining these expressions gives the result,
noting that the final statement does not involve $n$.
\end{proof}

\section{Random hypergraph preliminaries}\label{sec_rg}

We start with a trivial lower bound on the number $L_1$ of vertices in the largest component
of the random hypergraph $\Hrnp$ defined in Section~\ref{sec_prob}.
Throughout, $r\ge 2$ is fixed. Recall 
that $p=p(n) = d \frac{(r-1)!}{n^{r-1}}$ where $d=d(n)\to\infty$ and $\log n-d\to\infty$.
Set
\begin{equation}\label{s1def}
 s_1 = s_1(n) =100\max\{ n e^{-d},\log n\},
\end{equation} 
ignoring the irrelevant rounding to integers. Note for later that
\begin{equation}\label{ds1}
  d s_1 = o(n),
\end{equation}
since $de^{-d}\to 0$ and $d<\log n$ for $n$ large enough.

\begin{lemma}\label{nocut}
Let $d=d(n)\to\infty$ with $\log n-d\to\infty$,
and define $p$ as in \eqref{pdef} and $s_1$ as in \eqref{s1def}. Then
\[
 \Pr(L_1(\Hrnp) \le n-s_1)  = o(n^{-100}).
\]
\end{lemma}
\begin{proof}
Noting that $s_1\le n/4$, say, for $n$ large enough,
it is easy to see that if $L_1\le n-s_1$ then there is a vertex
cut $[n]=A\cup A^\cc$ with $s_1\le |A|\le n/2$ such that no hyperedge meets both $A$ and $A^\cc$.
For a given value of $a=|A|$, considering only potential hyperedges with one vertex in $A$ and the others
in $A^\cc$, the expected number of such cuts is crudely at most
\[
 \nu_a = \binom{n}{a} (1-p)^{a \binom{n-a}{r-1}} \le \left(\frac{en}{a}\right)^a \exp\left(-p a \binom{n-a}{r-1}\right).
\]
Now $\binom{n-a}{r-1}=\frac{n^{r-1}}{(r-1)!}(1+O(a/n))$.
Hence there is a constant $c>0$ such that for $a\le cn/d$ we
have $\binom{n-a}{r-1} \ge \frac{n^{r-1}}{(r-1)!}(1-1/d)$, say, and thus
\[
 \nu_a \le \left(\frac{en}{a}\right)^a\exp\left(-a\frac{pn^{r-1}}{(r-1)!}(1-1/d)\right)
 = \left(\frac{e^2ne^{-d}}{a}\right)^a \le e^{-2a}\le n^{-200},
\]
using $a\ge 100ne^{-d}$ and $a\ge 100\log n$ in the last two steps.
On the other hand, since $a\le n/2$, for $a\ge cn/d$
we still have $\binom{n-a}{r-1}\ge 2^{-r}\frac{n^{r-1}}{(r-1)!}$
and so
\[
 \nu_a \le \left(\frac{ en}{a} e^{-d/2^r}\right)^a \le \left(\frac{ed}{c} e^{-d/2^r}\right)^a
\le e^{-2a}
\]
if $n$ is large enough, recalling that $d=d(n)\to\infty$.
It follows easily that $\sum_{s_1\le a\le n/2}\nu_a= o(n^{-100})$, so the probability that there
is a cut as described is $o(n^{-100})$.
\end{proof}

We have shown that, up to a negligible error probability, the total size of all components with
at most $n/2$ vertices is at most $s_1$.
In particular, there are no individual components with more than $s_1$ vertices other than the unique
giant component.
We shall now show that with high probability the giant component is the only
component containing cycles. Furthermore, it is extremely unlikely that there is a 
cycle in a component of size between around $1000(\log n)/d$, say, and $n/2$.

Set
\begin{equation}
  s_0=\left\lceil\frac{1000 \log n}{d}\right\rceil \ge 1000.
\end{equation}
Recall that we assume that $\log n-d\to\infty$.
Since $de^{-d}$ is a decreasing function of $d$ for $d\ge 1$, it follows that
\begin{equation}\label{ndded}
 \frac{ n d e^{-d}}{\log n}\to \infty.
\end{equation}
Hence $s_0=o(s_1)$ and in particular, for $n$ large enough, $s_0<s_1$.

\begin{lemma}\label{nocyc}
Under the assumptions above, the probability that $\Hrnp$ contains a non-tree component
with at most $s_0$ vertices is $o(1)$.
Furthermore, if $X$ is the total number of edges in such components, then $\E[X^2]=o(1)$.
Finally, the probability that there is a non-tree component with more than $s_0$ but
fewer than $n/2$ vertices is $o(n^{-99})$.
\end{lemma}
\begin{proof}
We start with the second statement, which implies the first.
We aim for simplicity rather than a strong bound. If a non-tree component has $k$ edges and $v$
vertices, then $v\le (r-1)k$. Hence $k\ge v/(r-1)$.
Let $C_v$ be the number of $v$-element subsets $S$ of $[n]$ with the following properties:
no edge of $\Hrnp$ meets both $S$ and $S^\cc$, and $S$ spans at least $v/(r-1)$ edges of
$\Hrnp$. The vertex set of any non-tree component is such a set for some $v\ge r$.
Hence the number of non-tree components of $\Hrnp$ with at most $s_0$ vertices is at most
\[
 \sum_{v=r}^{s_0} C_v.
\]
Suppose that $e$ and $f$ are (not necessarily distinct) edges of $\Hrnp$ both in non-tree
components with at most $s_0$ vertices,
with vertex sets $S_1$ and $S_2$, say. Then $S_i$ spans at least $|S_i|/(r-1)$ edges
so (since the $S_i$ are equal or disjoint), $S=S_1\cup S_2$ has the properties above. Since a
set of $v$ vertices spans (very crudely) at most $v^r$ edges, we thus have
\[
 X^2 \le \sum_{v=r}^{2s_0} v^{2r} C_v.
\]

For $r\le v\le 2s_1$
let $k=\ceil{v/(r-1)}=\Theta(v)$. Very crudely,
\[
 \E[C_v] \le \binom{n}{v} \binom{\binom{v}{r}}{k} p^k (1-p)^{v\binom{n-v}{r-1}},
\]
where the last factor accounts for the fact that there are no edges consisting of one vertex
in the set $S$ and $r-1$ vertices outside. 
Using the very crude bound $\binom{v}{r}\le v^r$, and the slightly more careful bound
$\binom{n-v}{r-1}=\frac{n^{r-1}}{(r-1)!}(1+O(v/n))$ together with the inequality
$1-p\le e^{-p}$, it follows that
\begin{align*}
 \E[C_v] &\le \left(\frac{en}{v}\right)^v \left(\frac{e v^r}{k}\right)^k p^k
  \exp\left(-\frac{pvn^{r-1}}{(r-1)!} (1+O(v/n))\right) \\
  & = \frac{e^{v+k} n^v v^{rk}}{v^v k^k} d^k n^{-(r-1)k} (r-1)!^k \exp\bb{ -dv +O(dv^2/n) }.
\end{align*}
Since $r$ is constant and $v=\Theta(k)$, for some constant $B$ depending only on $r$ we thus have
\[
 \E[C_v] \le B^k \frac{n^{v-(r-1)k}}{k^{v-(r-1)k}} d^k \exp(-dv+O(dv^2/n)).
\]
Recall (from \eqref{ds1}) that $v\le 2s_1=o(n/d)=o(n)$. For $n$ large enough, in the range
$r\le v\le 2s_1$ we thus have
\begin{equation}\label{ECv}
 \E[C_v] \le B^k \frac{n^{v-(r-1)k}}{k^{v-(r-1)k}} d^k \exp(-dv/2).
\end{equation}
Since $(n/k)^{r-1}d^{-1}\ge (n/k)d^{-1} \ge n/(s_1d) \to\infty$,
if $v$ is not a multiple of $r-1$, then replacing $k=\ceil{v/(r-1)}$ by $k'=v/(r-1)$ in the right-hand
side of \eqref{ECv} can only increase it, so
\[
 \E[C_v] \le B^{v/(r-1)}d^{v/(r-1)} \exp(-dv/2) =((Bd)^{1/(r-1)} e^{-d/2})^v \le e^{-dv/4}
\]
if $n$ is large enough. Then
\[
 \sum_{v=r}^{2s_0} v^{2r} \E[C_v]
 \le \sum_{v=r}^\infty v^{2r} e^{-dv/4} \sim r^{2r} e^{-dr/4} \to 0,
\]
and the bound on $\E[X^2]$ follows.

For the final statement, simply note that
\[
 \sum_{v=s_0}^{s_1} \E[C_v] \le \sum_{v=s_0}^\infty  e^{-dv/4} \sim e^{-ds_0/4} =o(n^{-100})
\]
by choice of $s_0$, and apply Lemma~\ref{nocut} to deal with sizes between $s_1$ and $n/2$.
\end{proof}

\subsection{Small tree components}

Let $T_k=T_k(\Hrnp)$ denote the number of components of $\Hrnp$ that are trees with $k$ edges,
and so $s=1+(r-1)k$ vertices. Then $sT_k$ is the number of vertices in such components.
We already know that with very high probability there are no components with more than $s_1$
vertices other than the giant component; we shall see that with very high probability there 
are no tree components with more than $s_0$ vertices.
Recalling our convention of writing $s=1+(r-1)k$, set
\[
 k_0 = \frac{s_0-1}{r-1},
\]
ignoring the rounding to integers. Define $\pi_k$ as in Section~\ref{ss_pp}.
\begin{lemma}\label{lEtk}
For $0\le k=k(n)\le k_0$ we have
\[
 \E[T_k]  = n\frac{\pi_k}{s}(1+O(d^2s^2/n)),
\]
while the expected number of tree components with between $s_0$ and $n/2$
vertices is $o(n^{-99})$.
\end{lemma}
\begin{proof}
Let $k\ge 0$ and set $s=1+(r-1)k$.
We assume throughout the rather weak bound $s\le s_1$; Lemma~\ref{nocut}
implies that the expected number of tree components with between $s_1$ and $n/2$
vertices is $o(n^{-99})$.
By a result of Selivanov~\cite{Selivanov} (see also~\cite{KL_sparse,smoothing}), there are exactly 
\[
 n_k =  s^{k-1} \frac{(s-1)!}{k!\, (r-1)!^k}
\]
$k$-edge $r$-trees on a given set of $s$ vertices.
Clearly,
\[ 
 \E[T_k]  = \binom{n}{s} n_k p^k (1-p)^{\binom{n}{r}-\binom{n-s}{r}+k}.
\]
Since $\binom{n}{s} = \frac{n^s}{s!} \exp(O(s^2/n))$, we have
\[
 \binom{n}{s} n_k = \frac{n^s s^{k-2}}{k!(r-1)!^k} \exp(O(s^2/n)).
\]
Recalling that $p=d(r-1)!n^{-(r-1)}$, it follows that
\begin{equation}\label{ETk1}
 \E[T_k]  = \frac{n^{s-(r-1)k} s^{k-2}}{k!} d^k (1-p)^{\binom{n}{r}-\binom{n-s}{r}+k} \exp(O(s^2/n)).
\end{equation}

Now $\binom{n}{r}-\binom{n-s}{r} = s n^{r-1}/(r-1)! +O(s^2n^{r-2})$.
Since $k\le s=O(s^2n^{r-2})$, we thus have
\[
 \binom{n}{r}-\binom{n-s}{r}+k = \frac{sn^{r-1}}{(r-1)!}(1+O(s/n)).
\]
Since $p=O(d/n^{r-1}) = O(d/n)$, we have $-\log(1-p)=p+O(p^2)=p(1+O(d/n))$.
Thus
\begin{multline}\label{factor}
 (1-p)^{\binom{n}{r}-\binom{n-s}{r}+k} =  \exp\left(- p\frac{ sn^{r-1}}{(r-1)!} (1+O((d+s)/n)) \right) \\
  = \exp(-ds(1+O(d+s)/n))   = \exp(-ds+O(d^2s^2/n)),
\end{multline}
where in the second step we used again the definition $p=d(r-1)!n^{-(r-1)}$, and in the last
step we bounded $d+s$ by $ds$ just to keep the formula compact.
Since $s-(r-1)k=1$, combining \eqref{ETk1} and \eqref{factor} we obtain
\begin{equation}\label{ETk}
 \E[T_k] = n \frac{s^{k-2}d^k}{k!}\exp(-ds +O(d^2s^2/n)) = n\frac{\pi_k}{s} \exp(O(d^2s^2/n)),
\end{equation}
using Lemma~\ref{pkbd} in the last step.

For $s\le s_0$ we have $d^2s^2/n=O(d^2s_0^2/n) = O((\log n)^2/n)=o(1)$, so we may write the $\exp(O(d^2s^2/n))$
error term as $1+O(d^2s^2/n)$, giving the result.

For $s_0<s\le s_1$, from \eqref{ETk}, Corollary~\ref{decay}
and the bound $ds/n\le ds_1/n=o(1)$ (see \eqref{ds1}), we have
\[
 \E[T_k] \le n\exp(-ds/2+O(d^2s^2/n)) \le n\exp(-ds/4)
\]
if $n$ is large enough. Since $ds\ge ds_0\ge 1000\log n$, this completes the proof.
\end{proof}

\begin{corollary}\label{Csmall}
With probability $1-o(n^{-98})$ the hypergraph $\Hrnp$ consists of a `giant'
component with at least $n-s_1> n/2$
vertices together with `small' components, each with at most $s_0$ vertices.
\end{corollary}
\begin{proof}
Immediate from Lemmas~\ref{nocut}, \ref{nocyc} and \ref{lEtk}.
\end{proof}

\section{Key parameters}\label{sec_means}

As in the previous section, fix $r\ge 2$,
let $d=d(n)\to\infty$ with $\log n-d\to\infty$, as in \eqref{dconds},
and set $p=p(n)=d\frac{(r-1)!}{n^{r-1}}$.
Define $\xi=\xi(d)$ by \eqref{xieqn}, recalling that $\xi\sim e^{-d}$ as $d\to\infty$.
Let $L_1$ and $M_1$ denote the numbers of vertices and edges in the largest component of
the random hypergraph $\Hrnp$,
chosen according to any rule if there is tie. (With high probability there will not be a tie.)
\begin{lemma}\label{EL1}
Under the assumptions above we have
\[ 
 \E[L_1] = n(1-\xi) + o(1)
\]
and
\[
 \Var[L_1] \sim \mu_0 = ne^{-d}.
\]
\end{lemma}
Recall that $\mu_0=\mu_0(n)=n e^{-d}$ is roughly the expected number of isolated vertices
in $\Hrnp$. Since $\xi\sim e^{-d}$, the lemma says that isolated vertices give the dominant contribution to
$\E[n-L_1]$, and (roughly speaking) that the Poisson-type distribution of the number of isolated vertices
is the dominant contribution to $\Var[L_1]$. The bound $\E[L_1]=n-\mu_0+o(\mu_0)$ would
not be precise enough when we come to apply our local limit theorem; we need a bound with an error
that is $o(\sqrt{\mu_0})$.
\begin{proof}
Let $v_T$ and $v_C$ denote the number of vertices in small tree and cyclic components, respectively,
where `small' means with at most $s_0$ vertices. By 
Corollary~\ref{Csmall}, with probability $1-o(n^{-98})$ we have
\begin{equation}\label{L1split}
 L_1 = n - v_T - v_C.
\end{equation}
Since all relevant quantities are bounded by $n$, it follows easily that
\begin{equation}\label{EL1split}
 \E[L_1] = n - \E[v_T] - \E[v_C] +o(n^{-97}) = n-\E[v_T] +o(1),
\end{equation}
using Lemma~\ref{nocyc}. Now
\[
 v_T = \sum_{k=0}^{k_0} s T_k
\]
where $T_k$ is the number of $k$-edge tree components, we write $s$ for $1+(r-1)k$ as usual,
and $k_0=(s_0-1)/(r-1)$. (We ignore the irrelevant rounding to integers.)
By Lemma~\ref{lEtk},
\[
 \E[v_T] = \sum_{k=0}^{k_0} n \pi_k (1+O(d^2s^2/n)).
\]
Now $\pi_0=e^{-d}$ and, by Corollary~\ref{decay}, $\pi_k\le e^{-d(s+1)/2}$.
Hence
\begin{equation}\label{s2tail}
 \sum_{k=0}^\infty \pi_k s^2 = e^{-d}+ O(e^{-d(r+1)/2}) \sim e^{-d}.
\end{equation}
Also $\sum_{k>k_0} \pi_k =o(n^{-99})$.
It follows that
\begin{equation}\label{EvT}
 \E[v_T] = n \sum_{k=0}^\infty \pi_k +O(d^2e^{-d}) = n\xi +o(1),
\end{equation}
which, with \eqref{EL1split} proves the first statement of the lemma.

Turning to the variance, from \eqref{L1split} we have
\begin{equation}\label{VL1}
 \Var[L_1] = \Var[v_T+v_C] +o(1) = \Var[v_T] + \Var[v_C] + \Covar[v_C,v_T] +o(1).
\end{equation}
Let $T_{k,k'}$ denote the number of ordered pairs of distinct tree components where the
first has $k$ edges and the second $k'$. Writing $s=1+(r-1)k$ and $s'=1+(r-1)k'$, and considering
separately pairs of vertices in the same or distinct tree components, we have
\begin{equation}\label{EvT2}
 \E[v_T^2] = \E\left[\left(\sum_{k=0}^{k_0} sT_k\right)^2\right]
 = \sum_{k\le k_0} s^2\E[T_k] + \sum_{k,k'\le k_0} ss'\E[T_{k,k'}].
\end{equation}
By Lemma~\ref{lEtk} again,
\begin{equation}\label{ssk}
 \sum_{k\le k_0} s^2\E[T_k] = n\sum_{k=0}^{k_0} s\pi_k (1+O(d^2s^2/n)) \sim n\pi_0 = \mu_0,
\end{equation}
using the rapid decrease of the $\pi_k$ for the last approximation.
On the other hand, writing $m_{s,s'}$ for the number of potential hyperedges that meet
both a given set of $s$ vertices and a given disjoint set of $s'$ vertices, we have
\[
 \E[T_{k,k'}] = \E[T_k] \E[T_{k'}] (1-p)^{-m_{s,s'}}.
\]
Since $m_{s,s'} \le ss'n^{r-2}$, $p=O(d n^{-r+1})$, and $ss'd=o(n)$ for $s$, $s'\le s_0$, it follows that
\[
 ss' \E[T_{k,k'}] =s \E[T_k]s' \E[T_{k'}] (1+O(dss'/n)) = n^2\pi_k\pi_{k'}(1+O(d^2(s+s')^2/n)),
\]
by Lemma~\ref{lEtk}.
Arguing as for \eqref{EvT} above, it follows that
\begin{multline}\label{ss'k}
 \sum_{k,k'\le k_0} ss'\E[T_{k,k'}] = n^2 \sum_{k,k'\le k_0} \pi_k\pi_{k'} + O(n d^2e^{-2d}) \\
 = n^2\xi^2 + o(ne^{-d})  = n^2\xi^2+o(n\xi).
\end{multline}
Indeed, from the rapid decay of $\pi_k$ as $k$ increases, the dominant contribution to the error
term is from the case $k=k'=0$; this contribution is $O(n^2\pi_0^2d^2/n) = O(ne^{-2d}d^2)$.

Putting the pieces together, from \eqref{EvT2}, \eqref{ssk}, \eqref{ss'k}, \eqref{EvT}
and the fact that $n\xi\sim ne^{-d}=\mu_0$, it follows that
\begin{multline}\label{VvT}
 \Var[v_T^2] = \E[v_T^2] - \E[v_T]^2 = (1+o(1))\mu_0 + n^2\xi^2 +o(n\xi) - (n\xi+o(1))^2 \\
 = (1+o(1))\mu_0. 
\end{multline}
It remains only to note that from Lemma~\ref{nocyc} we have $\Var[v_C]\le \E[v_C^2] = o(1)$
and hence $\Covar[v_C,v_T]\le (\Var[v_C]\Var[v_T])^{1/2}=o(\mu_0)$.
Then, recalling \eqref{VL1}, the result follows.
\end{proof}

\begin{lemma}\label{EM1}
We have
\[
 \E[M_1] = \frac{d(1-\xi^r)}{r} n +O(d).
\]
\end{lemma}
\begin{proof}
Calling a component `small' if it has at most $s_0$ vertices,
let $e_T$ be the number of edges in small tree components
and $e_C$ the number in small cyclic components. By Corollary~\ref{Csmall}, with
probability $1-o(n^{-98})$ we have
\begin{equation}\label{M1split}
 M_1 = e(\Hrnp) - e_T - e_C.
\end{equation}
Since $\E[e(\Hrnp)]=p\binom{n}{r} = dn/r+O(d)$, it follows that
\begin{equation}\label{EMsplit}
 \E[M_1] = dn/r - \E[e_T] - \E[e_C] +O(d).
\end{equation}
Now
\begin{equation}\label{ecb}
 \E[e_C]\le \E[e_C^2] = o(1)
\end{equation}
by Lemma~\ref{nocyc}. On the other hand, writing $s=1+(r-1)k$ as usual, and setting $k_0=(s_0-1)/(r-1)$,
\[
 \E[e_T] = \sum_{k=1}^{k_0} k \E[T_k] = \sum_{k=0}^{k_0} \frac{kn}{s} \pi_k(1+O(d^2s^2/n)),
\]
by Lemma~\ref{lEtk}. Using Corollary~\ref{decay} as before to bound both the tail
of the sum and the contribution from the $O(d^2s^2/n)$ term (see~\eqref{s2tail}),
it follows that
\begin{multline*}
 (r-1) \E[e_T] = n \sum_{k=0}^\infty \frac{(r-1)k}{s} \pi_k + o(1) \\
 = n \sum_{k=0}^\infty \frac{s-1}{s} \pi_k + o(1) = \bb{ \xi  -\E[|\bp_{r,d}|^{-1}] }n +o(1),
\end{multline*}
since $\xi=\Pr(|\bp_{r,d}|<\infty)=\sum_k \pi_k$.

Recall from Section~\ref{sec_dual}
that the conditional distribution of $|\bp_{r,d}|$ given that it is finite is exactly
the distribution of $|\bp_{r,d^*}|$, where $d^*=d\xi^{r-1}$ is the dual parameter,
as in \eqref{dstardef}.
It follows by Lemma~\ref{negmo} that
\[
  \E[|\bp_{r,d}|^{-1}] = \Pr(|\bp_{r,d}|<\infty) \E[|\bp_{r,d^*}^{-1}|] = \xi (1-(r-1)d^*/r).
\]
Thus
\begin{equation}\label{EeT}
 \E[e_T] = \frac{\xi - \E[|\bp_{r,d}|^{-1}]}{r-1}n  +o(1) = \frac{\xi d^*}{r}n +o(1).
\end{equation}
From \eqref{EMsplit}, \eqref{ecb} and \eqref{EeT} we have
\[
 \E[M_1] = \frac{d-\xi d^*}{r} n +O(d) = \frac{d(1-\xi^r)}{r} n +O(d),
\]
completing the proof.
\end{proof}

\begin{lemma}\label{VM1}
We have
\[
 \Var[M_1] \sim \frac{dn}{r}.
\]
\end{lemma}
\begin{proof}
Clearly, $\Var[e(\Hrnp)] = \binom{n}{r}p(1-p) \sim p\binom{n}{r} \sim \frac{dn}{r}$.
Define $e_T$ and $e_C$ as in the proof of Lemma~\ref{EM1}. We claim that 
\begin{equation}\label{VeTaim}
 \Var[e_T] = o(dn).
\end{equation}
Let us first show that this implies the result. Indeed, from Lemma~\ref{nocyc} we have
$\Var[e_C]\le \E[e_C^2]=o(1)=o(dn)$. Hence $\Var[e_T+e_C]=o(dn)$, and from \eqref{M1split}
it easily follows that $\Var[M_1]\sim \frac{dn}{r}$. 

To establish \eqref{VeTaim} we argue as in the proof of \eqref{VvT}, using $e_T=\sum_{k\le k_0} kT_k$
in place of $v_T=\sum_{k\le k_0}sT_k$. The argument is essentially identical, leading
to the conclusion
$\Var[e_T] = \E[e_T] + o(n\xi)$.
Since $\xi$ and $d^*$ are $o(1)$, from \eqref{EeT} we have $\E[e_T]=o(n)$.\
Hence $\Var[e_T] =o(n)=o(dn)$.
\end{proof}

\section{Proof of Theorem~\ref{thLLT}}\label{sec_pf}

In this section we prove Theorem~\ref{thLLT}. This will require some further preparation.

Let $d=d(n)$ satisfy $d\to\infty$ and $\log n-d\to\infty$ as $n\to\infty$.
Note that by \eqref{ndded}, $n d e^{-d}/\log n\to \infty$.
Later, in various error terms we shall consider a function $\gamma(n)$ tending to zero slowly:
pick $\gamma=\gamma(n)$ such that
\begin{equation}\label{gammarate}
 \gamma\to 0,\quad \gamma d\to\infty \hbox{\quad and\quad}  \frac{\gamma^2 n d e^{-d}}{\log n}\to \infty
\end{equation}
as $n\to\infty$.

As before, let $p=d (r-1)! n^{-r+1}$. 
Set $d_1=\sqrt{d}$, and choose $d_2$ so that $p_1+p_2-p_1p_2=p$
where $p_i=d_i (r-1)! n^{-r+1}$. Note that
\begin{equation}\label{d1d2}
 d_1\to\infty,\quad d_2\sim d \hbox{\quad and\quad} d^2e^{-d_1} \to 0.
\end{equation}
Since $p_1+p_2=p(1+O(dn^{-r+1}))=p(1+O(d/n))$, we have
\begin{equation}\label{dsum}
 d_1+d_2 = d+O(d^2/n) = d+o(\gamma),
\end{equation}
since $d\le \log n$ for $n$ large and so (since $\gamma d\to\infty$) $\gamma\ge 1/\log n$.

Let $H_1$ and $H_2$ be independent random hypergraphs on the same vertex set $[n]$, with $H_i$
having the distribution of $\Hrnpi$. Clearly, $H=H_1\cup H_2$ has the distribution of $\Hrnp$.
We shall call the edges of $H_1$ \emph{red} and those of $H_2$ \emph{blue}. Note that there
may be a red and a blue edge on the same set of $r$ vertices.

The idea of the proof is as follows: we shall reveal the graph $H_1$ and some partial information
about $H_2$. We write the pair $(L_1,M_1)$ as $(L,M)+(X,Y)$ where $(L,M)$ is determined by the
revealed information, and the conditional distribution of $(X,Y)$ is with very high probability
a fixed, very simple distribution. The latter distribution (essentially two
independent binomial random variables)
will satisfy a local limit theorem. We will also have $\Var[X]\sim \Var[L_1]$ and $\Var[Y]\sim \Var[M_1]$.
This will easily imply that $L$ and $M$ are concentrated on the relevant scales, allowing us to 
transfer the local limit theorem to $(L_1,M_1)$. Related smoothing ideas were used in \cite{C-OMS,BC-OK1,smoothing},
though in a much more complicated way -- there, part of the starting point was a central
limit theorem. Here we do not need this, since our `smoothing distribution' $(X,Y)$ has asymptotically
the entire variance of the original distribution.

Roughly speaking, the partial information will be as follows: we reveal $H_1$ and find with very
high probability a very large connected component. We reveal all edges of $H_2$ except
those of the following form: ones within the giant component of $H_1$ (`internal edges' below),
and ones consisting of $r-1$ vertices in this giant component and one vertex that is otherwise
isolated (`peripheral edges'). Then $X$ and $Y$ will be, roughly speaking, the numbers of
peripheral and internal edges present. To obtain fixed distributions for $X$ and $Y$ we work
with a subset of the giant component of $H_1$ of a fixed size $b=b(n)$, and a fixed number $i=i(n)$
of vertices outside on which we allow peripheral edges. We will need to show that with
very high probability $L_1(H_1)\ge b$, and that there are enough of these outside vertices.

Turning to the details,
note that since $d_1=\sqrt{d}\le\sqrt{\log n}$ we have $e^{-d_1}=n^{-o(1)}$, so
\begin{equation}\label{nd1}
 n e^{-d_1} = n^{1-o(1)}.
\end{equation}
By analogy with \eqref{s1def}, but replacing $d$ by $d_1$, define $s_{1,1} = 100 n e^{-d_1}$ and set
\[
 b = \ceil{n-s_{1,1}} \sim n.
\]
Let $\cG_1$ be the `good' event
\[ 
 \cG_1 = \{ L_1(H_1) \ge b \}.
\]
By Lemma~\ref{nocut}, applied with $d_1$ in place of $d$, we have
\begin{equation}\label{cG1} 
 \Pr( \cG_1^\cc)  = o(n^{-100}).
\end{equation}
Whenever $\cG_1$ holds, let $B=B(H_1)$ be a set of $b$ vertices from the largest component
of $H_1$, say the first $b$ in numerical order. When $\cG_1$ does not hold, we take $B=\emptyset$.
(This is just so that $B$ is always defined; we will never use $B$ in this case.)

Let $\cG_2$ be the event that $e(H_1)\le n^{3/2}$, say. Since,
crudely, $\E[e(H_1)]=p_1\binom{n}{r}=O(n\log n)$, and
$e(H_1)$ has a binomial distribution, we certainly have
\begin{equation}\label{cG2}
 \Pr(\cG_2^\cc)=o(n^{-100}).
\end{equation}
Set
\begin{equation}\label{Mdef}
 \MM = \binom{b}{r} - \ceil{n^{3/2}} \sim \binom{n}{r}.
\end{equation}
When $\cG_1\cap \cG_2$ holds, we select a set $E$ of $r$-element subsets of $B$, none of which
is an edge of $H_1$, with $|E|=\MM$.
When $\cG_1\cap \cG_2$ does not hold, set $E=\emptyset$.
We call an edge $e$ of $H_2$ \emph{internal} if $e\in E$.

As a first step towards defining `peripheral' edges,
let us call a vertex $v\notin B$ \emph{peripheral} if either $v$ is isolated in $H$, or $v$ is in precisely
one edge $e$, and that edge $e$ is blue and consists of $v$ together with $r-1$ vertices in $B$.
Let $P_0$ be the set of peripheral vertices. Let
\begin{equation}\label{idef}
 i = \ceil{(1-\gamma) n d_2 e^{-d}},
\end{equation}
let $\cG_3$ be the event
\[
 \cG_3= \{ |P_0| \ge i\},
\]
and set
\[
 \cG = \cG_1\cap \cG_2\cap \cG_3.
\]
When $\cG$ holds, let $P$ consist of the first $i$ vertices in $P_0$ in some arbitrary order;
otherwise, set $P=\emptyset$.
We call an edge $e$ of $H_2$ \emph{peripheral} if it consists of 
one vertex in $P$ and $r-1$ vertices in $B$.
Given the pair $(H_1,H_2)$, let $\Ei$ be the set of internal blue edges, and let
$\Ep$ be the set of peripheral blue edges.

\begin{figure}[!t]
\centering
\[ \epsfig{file=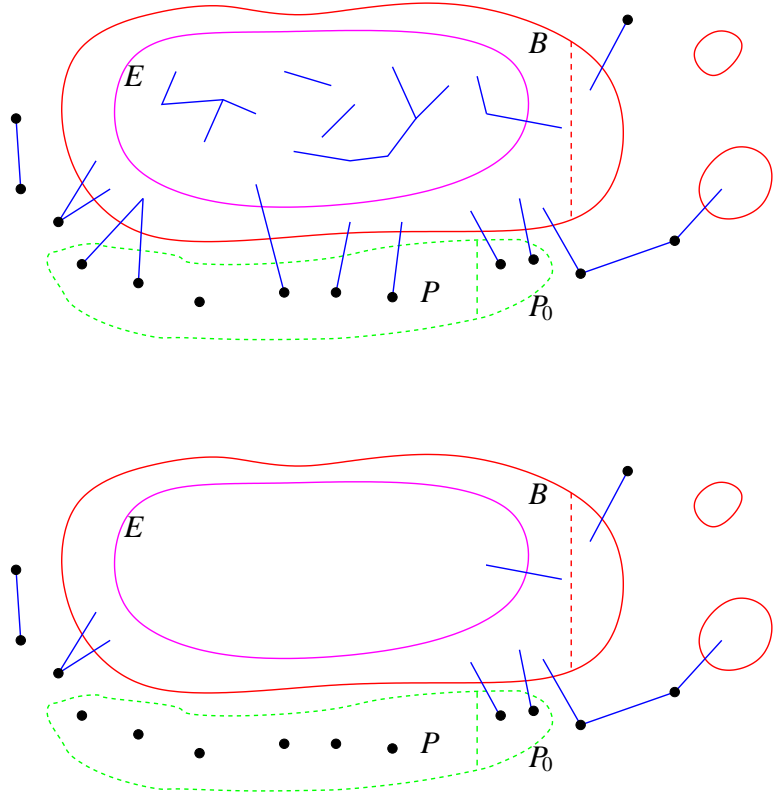,width=4in} \]
\caption{An example of the coloured graph $(H_1,H_2)$ (top) and the corresponding reduced graph $(H_1,H_2^-)$
(bottom), assuming $\cG$ holds.
The red blobs indicate components of $H_1$; vertices isolated in $H_1$ are shown as black dots. One
component of $H_1$ contains almost all vertices. $B$ is a large subset of this component with $|B|=b$.
$E$ indicates a set of potential edges, all inside $B$. Edges of $H_2$ are blue; there are many more
than are shown.
$P_0$ consists of those vertices $v$ that are isolated in $H_1$ and in at most one blue edge, with that
edge consisting of $v$ and $r-1$ vertices in $B$. $P$ is a subset of $P_0$ of a fixed size $i$. In reducing
$(H_1,H_2)$ we delete all blue edges in $E$, and all blue edges between $P$ and $B$.}\label{figR}
\end{figure}

Define the \emph{reduced (blue) hypergraph} to be
\[
 H_2^-=H_2-\Ei-\Ep,
\]
and let $H^-=H_1\cup H_2^-$; see Figure~\ref{figR}. Note that any $v\in P$ is isolated in $H_2^-$.
We shall condition on the pair $(H_1,H_2^-)$.

\begin{remark}\label{Rreadoff}
A key point is that we can determine $B$, $E$, $P_0$ and $P$, and hence whether $\cG$ holds,
knowing only the reduced graph $(H_1,H_2^-)$,
without knowing the original value of $H_2$. For $B$ and $E$ this is immediate; they are defined
in terms of the red graph $H_1$. With $H_1$ and hence $B$ fixed, for any possible value $H_2'$ of $H_2$,
let us temporarily write $P(H_1,H_2')$
for the set of peripheral vertices in $(H_1,H_2')$. Thus $v\in P(H_1,H_2')$ if and only if $v$
is isolated in $H_1$, is in at most one edge of $H_2'$ of the form $\{v,b_1,\ldots,b_{r-1}\}$
with $b_i\in B$, and is in no other edges of $H_2'$. The key observation is that
deleting internal or peripheral blue edges does not change $P(H_1,\cdot)$.
Thus $P_0=P(H_1,H_2)=P(H_1,H_2^-)$ is a function of $H_1$ and $H_2^-$ as claimed.
Since $P$ is defined in terms of $P_0$ only, $P$ is also a function of $H_1$ and $H_2^-$.
\end{remark}

\begin{lemma}\label{adds}
If $n$ is large enough, then whenever $\cG$ holds we have
\[
 L_1 = L + |\Ep|\hbox{\quad and\quad}M_1=M+|\Ep|+|\Ei|
\]
for some quantities $L$, $M$ that are functions of $(H_1,H_2^-)$.
\end{lemma}
\begin{proof}
Recall that $H=H_1\cup H_2=H_1\cup H_2^-\cup \Ep\cup \Ei$.
When $\cG$ holds, then $B$ is a set of $b$ vertices all within the same component
of $H_1$. Let $C^-$ be the component of $H^-=H_1\cup H_2^-$ containing this component.
Then $|C^-|\ge b>n/2$ if $n$ is large enough, so $C^-$ is certainly a subset of the largest
component $C$ of $H=H_1\cup H_2$. Furthermore, each edge of $\Ei$ lies entirely within $C^-$,
and each edge in $\Ep$ connects a distinct isolated vertex of $H^-$ to $C^-$, so we have
$|C|=|C^-|+|\Ep|$ and $e(C)=e(C^-)+|\Ep|+|\Ei|$.
\end{proof}

Since $b=n-s_{1,1}=n(1-O(e^{-d_1}))$, we have
\begin{equation}\label{p2br}
 p_2\binom{b}{r-1} = p_2\frac{n^{r-1}}{(r-1)!}\bb{1+O(s_{1,1}/n)} = d_2 (1+O(e^{-d_1})).
\end{equation}
This estimate will be useful in the proofs of the next two lemmas.

Recall from Remark~\ref{Rreadoff} that knowing the reduced graph $(H_1,H_2^-)$ 
determines $B$, $E$, $P_0$ and $P$, and hence also whether or not $\cG$ holds.

\begin{lemma}\label{cdist}
Whenever $\cG$ holds, the conditional distributions of $|\Ei|$ and $|\Ep|$ given $(H_1,H_2^-)$
are independent, with 
\[
 |\Ep|\sim \Bi(i,\pi) \quad\hbox{and}\quad  |\Ei|\sim \Bi(\MM,p_2),
\]
where
\begin{equation}\label{pidef}
 \pi = \frac{p_2 \binom{b}{r-1}}{p_2 \binom{b}{r-1} + 1-p_2} \sim \frac{d_2}{1+d_2}.
\end{equation}
\end{lemma}
\begin{proof}
Given $(H_1,H_2^-)$ it is easy to identify the possible values of the pair $(\Ep,\Ei)$.
Indeed, as noted above the pair $(H_1,H_2^-)$ determines $B$, $E$ and hence $P_0$ and $P$.
When $\cG$ holds, then $|B|=b$, $|E|=\MM$, $|P|=i$,
and $P$ and $B$ are disjoint.
Now $\Ei$ must be a subset of $E$, and any subset is possible.
A peripheral blue edge must consist of a vertex of $P$ and $r-1$ vertices of $B$, and any set of such edges
including each vertex of $P$ at most once is a possibility for $\Ep$. In other words, to
obtain a possible value of $H_2$, given $(H_1,H_2^-)$, starting from $H_2^-$ we must

(a) for each edge $e\in E$, either add it or not, and 

(b) for each vertex $v\in P$, either add one of the $\binom{b}{r-1}$ edges consisting of $v$ and $r-1$
vertices from $B$ or not.

Since all combinations are possible, and the probability of a possible value $H$ of $H_2$ is proportional to $p_2/(1-p_2)$
to the power of the number of edges, we see that in describing the conditional distribution of $H_2$ given $(H_1,H_2^-)$,
all these choices are independent. Furthermore, in (a) each edge is included with probability $p_2$, and in (b)
the probability of including an edge is $\pi$ as defined above. This implies the claimed formulae,
with the asymptotic estimate for $\pi$ following from \eqref{p2br}.
\end{proof}

Our `smoothing random variables' will be essentially $|\Ep|$ and $|\Ei|$. For later, it will be convenient
to `cook' these random variables when $\cG$ does not hold, so that they \emph{always} have the conditional
distribution defined above. More precisely, let
\[
  X'\sim \Bi(i,\pi) \quad\hbox{and}\quad  Y'\sim \Bi(\MM,p_2)
\]
be independent of each other and of $(H_1,H_2^-)$.
Set
\begin{equation}\label{XYdef}
 X = \ind{\cG} |\Ep| + \ind{\cG^\cc} X' \quad\hbox{and}\quad 
 Y = \ind{\cG} |\Ei| + \ind{\cG^\cc} Y',
\end{equation}
where $\ind{A}$ denotes the indicator function of an event $A$.
The only properties of $(X,Y)$ we shall need are the following.
\begin{lemma}\label{XY}
The random variables $(H_1,H_2^-)$ and $(X,Y)$ are independent, with $(X,Y)$ having the distribution
of a pair of independent binomial random variables $\Bi(i,\pi)$ and $\Bi(\MM,p_2)$.
Moreover, if $n$ is large enough, then whenever $\cG$ holds we have
\[
 L_1 = L + X\hbox{\quad and\quad}M_1=M+X+Y
\]
for some quantities $L$, $M$ that are functions of $(H_1,H_2^-)$.
\end{lemma}
\begin{proof}
To prove the first statement we must show exactly that the conditional distribution
of $(X,Y)$ given $(H_1,H_2^-)$ is always that of the given independent binomial
distributions. This follows immediately from the definition \eqref{XYdef} of $X$ and $Y$
and Lemma~\ref{cdist} (applied only when $\cG$ holds).
The second statement follows immediately from \eqref{XYdef} and Lemma~\ref{adds}.
\end{proof}

We have already shown that the events $\cG_1$ and $\cG_2$ are very likely to hold.
We now show the same for the event $\cG_3=\{P_0\ge i\}$ that there are at least
$i=\ceil{(1-\gamma) n d_2 e^{-d}}$ peripheral vertices.

\begin{lemma}\label{lcG3}
We have $\Pr(\cG_3) = 1-o(n^{-100})$.
\end{lemma}
\begin{proof}
The probability that a given vertex is isolated in $H_1$ is $(1-p_1)^{\binom{n-1}{r-1}}$.
Recalling that $p_1=O(n^{-r+1}\log n)$,
since $\binom{n-1}{r-1}=n^{r-1}/(r-1)! + O(n^{r-2})$ and $p_1^2 n^{r-1}=O(n^{-r+1}(\log n)^2)=O((\log n)^2/n)$,
we have (crudely)
\begin{multline*}
 (1-p_1)^{\binom{n-1}{r-1}} = \exp\bb{-p_1\tbinom{n-1}{r-1} +O((\log n)^2/n)} \\
 = \exp\bb{-p_1\tfrac{n^{r-1}}{(r-1)!} +O((\log n)^2/n)} = e^{-d_1+O((\log n)^2/n)} = e^{-d_1}(1+o(\gamma)),
\end{multline*}
since $\gamma n/(\log n)^2\ge \gamma d\to\infty$, from \eqref{gammarate}.
Let $I_1$ be the set of isolated vertices of $H_1$. Then $\E[|I_1|]=ne^{-d_1}(1+o(\gamma))$.
Since $\E[e(H_1)] = n d_1/r \le n\log n$, say, it is not hard to see that
\[
 \Pr\bb{ |I_1| \le \E[|I_1|] - n^{0.51} } = o(n^{-100}),
\]
say. For example, one approach is to consider a variant of the model $H_1=\Hrnpone$ in which we add
uniformly random edges one-by-one, and apply the Hoeffding--Azuma inequality in this model;
we omit the details. Recalling \eqref{nd1}, we see that with probability $1-o(n^{-100})$
we have
\begin{equation}\label{I1bd}
 |I_1| \ge (1-n^{-1/3}) \E[|I_1|] \ge n e^{-d_1} (1-o(\gamma)),
\end{equation}
say.

For the rest of the proof we condition on $H_1$, which determines $I_1$ and $B$; we assume, as we may,
that the event $\cG_1$ and inequality \eqref{I1bd} hold. 
Since $\cG_1$ holds, the sets $I_1$ and $B$ are disjoint.
Call a possible (blue) edge \emph{acceptable} if it consists of one vertex of $I_1$ and $r-1$ vertices
of $B$. We call any other possible edge meeting $I_1$ \emph{annoying}. It remains to show that
with very high probability there are at least $i$ vertices $v\in I_1$ such that 
$v$ is in no annoying blue edges and $v$ is in at most one acceptable blue edge.
Indeed, when $\cG_1$ holds, then $P_0$ consists precisely of the set of such $v\in I_1$.

We first consider annoying edges. There are at most
\[
 N = |I_1|(n-|B|)n^{r-2} = |I_1| s_{1,1} n^{r-2}
\]
possible edges that are annoying. Each is present in $H_2$ independently with probability $p_2$.
By \eqref{d1d2}, $d_2e^{-d_1}\le de^{-d_1}\to 0$, so if $n$ is large enough we have
\[
 Np_2 = |I_1| 100 n e^{-d_1} n^{r-2} d_2 \frac{(r-1)!}{n^{r-1}} = 100 (r-1)! d_2 e^{-d_1} |I_1| \le |I_1|/(2d).
\]
Recalling \eqref{nd1} and \eqref{I1bd}, we have $|I_1|/d\ge n e^{-d_1}/(2d) = n^{1-o(1)}$.
It follows by a Chernoff bound that with probability at least $1-\exp(-\Omega(|I_1|/d)) = o(n^{-100})$
the actual number of annoying blue edges is at most $|I_1|/d$.

We condition on the set of annoying edges present in $H_2$, assuming there are at most $|I_1|/d$ of them.

Let $I_1'\subseteq I_1$ be the set of vertices of $I_1$ not in any annoying blue edges, so
\begin{equation}\label{I1'bd}
 |I_1'| \ge |I_1| - r|I_1|/d = |I_1|(1-r/d) \ge n e^{-d_1} ( 1-o(\gamma)).
\end{equation}
A vertex $v\in I_1'$ is in $P_0$ if and only if it is in at most one acceptable blue edge.
Noting that we have not yet tested the acceptable edges for their presence in $H_2$,
for a given vertex $v\in I_1'$, this event has probability
\[
 \pi= (1-p_2)^{\binom{b}{r-1}} + \tbinom{b}{r-1}p_2(1-p_2)^{\binom{b}{r-1}-1} \ge
 \tbinom{b}{r-1}p_2(1-p_2)^{\binom{b}{r-1}}.
\]
Recall from \eqref{p2br} that $p_2\binom{b}{r-1} = d_2 (1+O(e^{-d_1}))$.
Since $p_2^2n^{r-1}=O(n^{-r+1}(\log n)^2)$, it easily follows that
\begin{equation}\label{pibd}
 \pi \ge d_2 e^{-d_2} (1-o(\gamma)).
\end{equation}
Since the sets of potential acceptable edges meeting different vertices in $I_1'$ are disjoint,
the events $v\in P_0$ are (conditionally) independent for different $v\in I_1'$, so the conditional distribution
of $|P_0|$ is binomial $\Bi(|I_1'|,\pi)$.
Now by \eqref{I1'bd}, \eqref{pibd} and \eqref{dsum} we have
\[
 |I_1'|\pi \ge n e^{-d_1} d_2 e^{-d_2} (1-o(\gamma)) \ge n d_2 e^{-d} (1-o(\gamma)).
\]
Since $\gamma^2 nd_2e^{-d}\sim \gamma^2 nd e^{-d}$ is much larger than $\log n$ by \eqref{gammarate},
a Chernoff bound shows that with (conditional) probability at least
$1-\exp(-\Omega(\gamma^2 |I_1'|\pi))=1-o(n^{-100})$ we have
\[
 |P_0|\ge (1-\gamma/2) |I_1'|\pi \ge (1-\gamma) n d_2 e^{-d},
\]
for $n$ large enough. Thus $\Pr(\cG^3)\ge 1-o(n^{-100})$, as claimed.
\end{proof}

\begin{corollary}\label{ccG}
We have $\Pr(\cG)=1-o(n^{-100})$, where $\cG=\cG_1\cap\cG_2\cap\cG_3$.
\end{corollary}
\begin{proof}
Immediate from \eqref{cG1}, \eqref{cG2} and Lemma~\ref{lcG3}.
\end{proof}

Before turning to the proof of Theorem~\ref{thLLT} we give a simple observation about
local limit theorems.  Recall that a sequence 
$((X_n,Y_n))$ of $\ZZ^2$-valued random variables satisfies a local limit
theorem (an LLT) with parameters $(\mu_X(n),\mu_Y(n))$ and $(\sigma_X^2(n),\sigma_Y^2(n))$ if, suppressing
the dependence on $n$, we have
\[ 
 \sup_{(x,y)\in \ZZ^2} \left| \Pr(X_n=x,Y_n=y)  - f_n(x-\mu_X,y-\mu_Y) \right| = o(1/(\sigma_X\sigma_Y)),
\]
where
\begin{equation}\label{fndef}
 f_n(x,y) = \frac{1}{2\pi\sigma_X\sigma_Y} \exp\left(-\frac{x^2}{2\sigma_X^2} -\frac{y^2}{2\sigma_Y^2} \right).
\end{equation}
We make some simple observations about such results.
\begin{lemma}\label{parshift}
If $(X_n,Y_n)$ satisfies an LLT with parameters $(\mu_X(n),\mu_Y(n))$ and $(\sigma_X^2(n),\sigma_Y^2(n))$,
then it also satisfies an LLT with parameters $(\tmu_X(n),\tmu_Y(n))$ and $(\tsigma_X^2(n),\tsigma_Y^2(n))$
whenever $\tmu_X(n)=\mu_X(n)+o(\sigma_X(n))$, $\tmu_Y(n)=\mu_Y(n)+o(\sigma_Y(n))$,
$\tsigma_X(n)\sim \sigma_X(n)$ and $\tsigma_Y(n)\sim\sigma_Y(n)$.
\end{lemma}
\begin{proof}
This is a standard result; we omit the proof which is just calculation.
\end{proof}

\begin{lemma}\label{smalloff}
Suppose that $(X_n,Y_n)$ satisfies an LLT with parameters $(\mu_X(n),\mu_Y(n))$ and $(\sigma_X^2(n),\sigma_Y^2(n))$.
Suppose also that $(A_n,B_n)$ is independent of $(X_n,Y_n)$, and that
\[
 \Var[A_n] = o(\sigma_X^2)\hbox{\quad and\quad} \Var[B_n] = o(\sigma_Y^2).
\]
Then $(A_n+X_n,B_n+Y_n)$ satisfies an LLT with parameters
$(\E[A_n]+\mu_X(n),\E[B_n]+\mu_Y(n))$ and $(\sigma_X^2(n),\sigma_Y^2(n))$.
\end{lemma}
\begin{proof}
This is a simple consequence of the fact that $A_n$ and $B_n$ are concentrated on the relevant scales, namely
$\mu_X=\mu_X(n)$ and $\mu_Y=\mu_Y(n)$, together with the fact that $\exp(-x^2/2)$ is Lipschitz as a function
of $x$ (with constant $e^{-1/2}$). Indeed, we may write
\begin{multline*}
 \Pr(A_n+X_n=x,B_n+Y_n=y) = \E_{A_n,B_n}\bigl[\Pr(X_n=x-A_n,Y_n=y-B_n)\bigr] \\
 =  \E_{A_n,B_n}\bigl[f_n(x-A_n-\mu_X,y-B_n-\mu_Y)\bigr] + o(1/(\sigma_X\sigma_Y)),
\end{multline*}
where $f_n$ is defined as in \eqref{fndef} and in the second step we applied the LLT for $(X_n,Y_n)$.
Since $\E[|A_n-\E[A_n]|] \le \sqrt{\Var{A_n}} = o(\sigma_X)$ and similarly
$\E[|B_n-\E[B_n]|] = o(\sigma_Y)$, the result follows from the Lipschitz property of $e^{-x^2/2}$.
\end{proof}

As in Theorem~\ref{thLLT}, set
\[
 \sigma_L^2=\sigma_L^2(n) = n e^{-d} \hbox{\quad and\quad} \sigma_M^2 = \sigma_M^2(n) = \frac{dn}{r}.
\]
Most of the time we shall suppress the dependence on $n$ in the notation. The final
ingredient in our proof is that the `cooked' versions $X$ and $Y$
of the numbers $\Ep$ and $\Ei$ of peripheral and internal
blue edges satisfy an LLT.

\begin{lemma}\label{EpiLLT}
The random variables $X=X_n$ and $Y=Y_n$ defined in \eqref{XYdef} satisfy a local limit theorem with 
parameters $(\E[X],\E[Y])$ and $(\sigma_L^2,\sigma_M^2)$.
\end{lemma}
\begin{proof}
By Lemma~\ref{XY}, $X$ and $Y$ are independent of each
other, and have binomial distributions $\Bi(i,\pi)$ and $\Bi(\MM,p_2)$.
Recalling \eqref{idef} and \eqref{pidef} we have
\begin{equation}\label{VEp}
 \Var[X] = i \pi (1-\pi) \sim i (1-\pi) \sim i/d_2 \sim n e^{-d} = \sigma_L^2.
\end{equation}
Also, from \eqref{Mdef},
\begin{equation}\label{VEi}
 \Var[Y] = \MM p_2 (1-p_2) \sim \MM p_2 \sim \frac{n^r}{r!} d_2 \frac{(r-1)!}{n^{r-1}} = \frac{d_2n}{r}
 \sim \frac{dn}{r} = \sigma_M^2.
\end{equation}
It is well known (and easy to verify, for example from the formula
for a binomial coefficient) that a sequence of binomial random variables $\Bi(N_n,p_n)$ with
$\sigma^2_n=N_np_n(1-p_n)\to\infty$ satisfies a univariate local limit theorem
with parameters $N_np_n$ and $\sigma^2_n$. This statement extends in an obvious way to a pair of independent
binomial random variables; this extension, and Lemma~\ref{parshift}, give the result.
\end{proof}

We now have all the pieces in place to prove our local limit theorem.
\begin{proof}[Proof of Theorem~\ref{thLLT}]
We have already proved the estimates \eqref{ebds} and \eqref{vbds} for the mean and variance
of $L_1=L_1(\Hrnp)$ and $M_1=M_1(\Hrnp)$ in Lemmas~\ref{EL1}, \ref{EM1} and~\ref{VM1}.
It remains to prove the local limit theorem.
Set
\[
 L_1^* = L + X \hbox{\quad and\quad} M_1^*=M+X+Y,
\]
where $X$ and $Y$ are defined in \eqref{XYdef} and
$L$ and $M$ are as in Lemma~\ref{XY}. By that lemma we have $L_1=L_1^*$ and $M_1=M_1^*$
whenever $\cG$ holds. By Corollary~\ref{ccG} we have $\Pr(\cG^\cc)=o(n^{-100})$.
It thus suffices to show that $(L_1^*,M_1^*)$
satisfies a bivariate local limit theorem with parameters $(\E[L_1^*],\E[M_1^*])$ for the means
and $(\sigma_L^2,\sigma_M^2)$ for the variances.
Furthermore, since $\sigma_L^2=o(\sigma_M^2)$, setting $\tM_1=M_1^*-L_1^*$, this is equivalent to showing that 
$(L_1^*,\tM_1)$ satisfies a local limit theorem with parameters $(\E[L_1^*],\E[\tM_1])$ and
$(\sigma_L^2,\sigma_M^2)$.

By Lemma~\ref{XY}, $(X,Y)$ is independent of $(H_1,H_2^-)$.
Since $L$ and $M$ are determined by $H_1$ and $H_2^-$,
we see that $(X,Y)$ is independent of $(L,M)$,
and hence of $(L,M-L)$.
Now
\begin{equation}\label{sum}
 (L_1^*,\tM_1) = (L,M-L) + (X,Y),
\end{equation}
with the summands independent.
Since $L_1^*=L_1$ with probability $1-o(n^{-100})$, we have
\[
 \Var[L_1^*]\sim \Var[L_1]\sim ne^{-d} = \sigma_L^2,
\]
by Lemma~\ref{EL1}.
Similarly, $\Var[M_1^*]\sim \Var[M_1]\sim dn/r=\sigma_M^2$ by Lemma~\ref{VM1}.
Since $\sigma_L=o(\sigma_M)$ it follows that
\[
 \Var[\tM_1] \sim \sigma_M^2.
\]
From \eqref{VEp} and \eqref{VEi} we have 
\[
 \Var[X] \sim \sigma_L^2 \hbox{\quad and\quad} \Var[Y]\sim \sigma_M^2.
\]
From the independence in \eqref{sum}, we have $\Var[L_1^*]=\Var[L]+\Var[X]$,
which implies $\Var[L]=o(\sigma_L^2)$. Similarly, $\Var[M-L]=o(\sigma_M^2)$.
The result now follows from Lemma~\ref{EpiLLT}, the independence in \eqref{sum}, and Lemma~\ref{smalloff}.
\end{proof}

\section{Proof of Theorem~\ref{thenum}}\label{sec_enum}

Our aim in this section is to deduce our enumerative result, Theorem~\ref{thenum}, from the probabilistic
one, Theorem~\ref{thLLT}. We start with a lemma giving the asymptotic behaviour of the quantity
$\xi=\xi(\bd)$ appearing in Theorem~\ref{thenum}.

Recall that for $r\ge 3$ we define a function $\Phi_r$ on $(0,1)$ by
\begin{equation*}
 \Phi_r(\xi) = \frac{ \log(1/\xi) (1-\xi^r)}{(1-\xi^{r-1})(1-\xi)},
\end{equation*}
and that $\Phi_r$ is a decreasing bijection between $(0,1)$ and $(r/(r-1),\infty)$.
If $r\ge 3$ then
\begin{equation}\label{Phiasy}
 \Phi_r(\xi) = \log(1/\xi)(1+\xi +O(\xi^2)) \sim  \log(1/\xi) \hbox{\quad as\quad} \xi\to 0.
\end{equation}
Recall also that $F_r(\bd)$ is then defined by \eqref{Fdef} with $\xi=\Phi_r^{-1}(\bd)$.

\begin{lemma}\label{asy}
Fix $r\ge 3$. For $\bd>r/(r-1)$ let $\xi(\bd)=\Phi_r^{-1}(\bd)$. Then as $\bd\to\infty$ we have
\begin{equation}\label{xiasy}
 \xi = e^{-\bd} + \bd e^{-2\bd} + O(\bd^2 e^{-3\bd})
\end{equation}
and
\begin{equation}\label{Fasy}
  F_r(\bd) = e^{-\bd} + \frac{\bd+1}{2}e^{-2\bd} + O(\bd^2 e^{-3\bd}).
\end{equation}
For $r=2$, we have $\xi=e^{-\bd}+2\bd e^{-2\bd}+O(\bd^2 e^{-3\bd})$ and
$F_2(\bd) = e^{-\bd} + (\bd+1/2)e^{-2\bd} + O(\bd^2 e^{-3\bd})$.
\end{lemma}
\begin{proof}
Since $\Phi_r$ is a decreasing bijection from $(0,1)$ to $(r/(r-1),\infty)$,
as $\bd\to\infty$ we have $\xi=\Phi_r^{-1}(\bd)\to 0$.
Hence, for $r\ge 3$, from \eqref{Phiasy} we have
\begin{equation}\label{lxi}
 \log\xi = -\bd(1-\xi+O(\xi^2)).
\end{equation}
It follows (multiplying by $\xi$) that $\bd\xi\to 0$. Also, from \eqref{lxi},
\[
 \xi = e^{-\bd} e^{\bd\xi+O(\bd\xi^2)} = e^{-\bd} (1+\bd\xi +O(\bd^2\xi^2))
 = e^{-\bd}+\bd e^{-2\bd} + O(\bd^2e^{-3\bd}).
\]
From \eqref{Fdef}, for $r\ge 3$ we have
\[
 F_r(\bd) = -\bd(\xi+\xi^2/2) - (\xi+\xi^2)\log\xi + \xi+\xi^2/2+O(\xi^3\bd+\xi^3|\log\xi|+\xi^3).
\]
Since $|\log\xi|\sim\bd$ as $\bd\to\infty$, we may write the error term as $O(\bd\xi^3)$.
Substituting in \eqref{lxi}, it follows that
\begin{eqnarray*}
 F_r(\bd) &=& -\bd\xi -\bd\xi^2/2+\bd\xi+\bd\xi^2-\bd\xi^2+\xi+\xi^2/2+O(\bd^2\xi^3) \\
    &=& \xi -(\bd-1)\xi^2/2+O(\bd^2\xi^3).
\end{eqnarray*}
Substituting in \eqref{xiasy} gives \eqref{Fasy}. We omit the (similar) calculations for $r=2$.
\end{proof}

Recall that we write $C_r(s,m)$ for the number of connected $r$-uniform hypergraphs
on $[s]$, and $P_r(s,m)$ for the probability that an $m$-edge $r$-uniform hypergraph
on $[s]$ chosen uniformly at random is connected. Clearly,
$P_r(s,m)=C_r(s,m)/\binom{N}{m}$, where $N=\binom{s}{r}$.
Thus the asymptotic formulae \eqref{Pform} and \eqref{Cform} for $P_r(s,m)$ and $C_r(s,m)$
are equivalent modulo a calculation, which we now carry out.

\begin{lemma}\label{lPC}
Let $r\ge 2$ and let $m=m(s)=o(s^{4/3})$. If $r\ge 3$ then
\begin{equation}\label{Nm1}
 \binom{\binom{s}{r}}{m} \sim \frac{s^{rm}}{m!r!^m} e^{-(r-1)\bd/2},
\end{equation}
as $s\to\infty$, where $\bd=rm/s$. If $r=2$, then
\begin{equation}\label{Nm2}
 \binom{\binom{s}{r}}{m} \sim \frac{s^{rm}}{m!r!^m} e^{-(r-1)\bd/2-\bd^2/4} .
\end{equation}
\end{lemma}
\begin{proof}
Let
\[
 N=\binom{s}{r} = \frac{s(s-1)\cdots (s-r+1)}{r!} = \frac{s^r}{r!}e^{-\binom{r}{2}/s+O(s^{-2})}.
\]
Since $m=o(s^{4/3})=o(s^2)$, we have
\[
 N^m \sim \frac{s^{rm}}{r!^m} e^{-\binom{r}{2}m/s} = \frac{s^{rm}}{r!^m} e^{-(r-1)\bd/2}.
\]
Since $N=\Theta(s^r)$, if $r\ge 3$ then $m^2=o(N)$. Thus
$\binom{N}{m}\sim N^m/m!$, giving \eqref{Nm1}.

For \eqref{Nm2}, suppose that $r=2$. Then $N=s^2/2(1+O(1/s))$ and $m=\bd s/2$.
Since $m/N$ and $m^3/N^2$ are $o(1)$, and $\bd^2/s=O(m^2/s^3)=o(1)$, we have
\begin{multline*}
 \binom{N}{m} = \frac{N^m}{m!} \exp\bb{-m^2/(2N) +O(m/N+m^3/N^2)} \\
 \sim \frac{N^m}{m!} \exp\left(-\frac{\bd^2s^2}{4s^2}(1+O(1/s))\right) \sim \frac{N^m}{m!} e^{-\bd^2/4},
\end{multline*}
and \eqref{Nm2} follows.
\end{proof}

We are finally ready to deduce Theorem~\ref{thenum} from Theorem~\ref{thLLT}.

\begin{proof}[Proof of Theorem~\ref{thenum}]
Throughout, we consider a function $m=m(s)$ with $m/s\to\infty$; all asymptotics are as $s\to\infty$.
Much of the time we suppress the dependence on $s$ in the notation.
We write
\[
 \bd = \bd(s) = \frac{rm}{s}
\]
for the average degree of an $r$-uniform hypergraph with $s$ vertices and $m$ edges.

Let us first deal with a simple case: when $\log s-\bd$ is bounded above. Passing to a subsequence,
we may assume that either (i) $\log s-\bd\to c$ for some constant $c\in \RR$, or (ii) $\log s-\bd\to-\infty$.
Let $\Hrsm$ be a hypergraph on $[s]$ with $m$ edges, chosen uniformly at random from all such hypergraphs.
Let $X=X_s$ denote the number of isolated vertices in $\Hrsm$. In case (i), a simple calculation shows
that
\[ 
 \E[X] = s \binom{\binom{s-1}{r}}{m} {\binom{\binom{s}{r}}{m}}^{-1} \to e^c.
\]
(This is to be expected, since the probability that a vertex is isolated is asymptotically $e^{-\bd}$.)
Similarly, for any fixed $k$ the $k$th factorial moment satisfies
\[ 
 \E[X(X-1)\cdots (X-k+1)] = s(s-1)\cdots (s-k+1) \binom{\binom{s-k}{r}}{m} {\binom{\binom{s}{r}}{m}}^{-1} \to e^{kc}.
\]
Now by a standard result (see, e.g., Theorem 1.22 in~\cite{BB:RG2})
it follows that $X=X_s$ converges in distribution to a Poisson distribution with mean $e^c$
as $s\to\infty$, and in particular, that $\Pr(X=0)\to \exp(e^{-c})$. A very simple argument counting cuts
shows that in this range, with high probability $\Hrsm$ has no component of size between $2$ and $s/2$, so
the probability $P_r(s,m)$ that $\Hrsm$ is connected satisfies
\[
 P_r(s,m) =\Pr(X=0)+o(1) \sim \exp(-e^c) \sim \exp(-se^{-\bd}).
\]
Since $\bd=\log s+O(1)$, we have $s\bd e^{-2\bd}\to 0$, so by Lemma~\ref{asy} $sF_r(\bd)=se^{-\bd}+o(1)$.
Thus we have $P_r(s,m)\sim \exp(-sF_r(\bd))$, proving \eqref{Pform} in this case.
For case (ii), it follows from the above and monotonicity
that $P_r(s,m)\to 1$, which again agrees with \eqref{Pform}.
In both cases (i) and (ii), provided $m=o(s^{4/3})$ (which \eqref{Cform} assumes),
relation \eqref{Cform} follows immediately from \eqref{Pform} and Lemma~\ref{lPC}.

In proving Theorem~\ref{thenum}, passing to a subsequence, we may assume either that
$\log s-\bd$ is bounded above, or that $\log s-\bd\to\infty$. We have covered
the first case above.
From now on we thus assume that
\begin{equation}\label{bds}
 \bd \to \infty \hbox{\quad and\quad} \log s-\bd \to \infty
\end{equation}
as $s\to\infty$. Then $m=\bd s/r=O(s\log s)=o(s^{4/3})$, so by Lemma~\ref{lPC}
either of \eqref{Pform} and \eqref{Cform} implies the other. We shall prove \eqref{Cform}.

When $s$ is large enough, we have $\bd> r/(r-1)$; we assume this from now on.
Then there is a unique solution $\xi=\xi(s)$ to \eqref{xidef}, i.e., to
\begin{equation*}
 \Phi_r(\xi) = \bd = \frac{rm}{s}.
\end{equation*}
Since $\bd\to\infty$ as $s\to\infty$ we have $\xi\to 0$.
Set
\begin{equation}\label{dsdef}
 d=d(s) = \frac{\log(1/\xi)}{1-\xi^{r-1}},
\end{equation}
noting that $d(s)\to\infty$. Then $\xi$ and $d$ solve the equation~\eqref{xieqn} (which
is just \eqref{dsdef} rearranged). Also, set
\[
 \tn = \tn(s) = \frac{s}{1-\xi}.
\]
The reason for these choices is that then we have
\begin{equation}\label{tnhit1}
 \tn(1-\xi) = s
\end{equation}
and
\begin{equation}\label{tnhit2}
  \frac{d(1-\xi^r)}{r}\tn = \frac{(1-\xi)\Phi_r(\xi)}{r} \tn = \frac{\Phi_r(\xi)}{r}s = m.
\end{equation}
We would like to apply Theorem~\ref{thLLT} with the parameters $\tn$ and $d$ just defined;
one trivial but annoying difficulty is that $\tn$ is not an integer. So set
\[
 n = n(s) = \ceil{\tn}.
\]
Since $n=\tn+O(1)$, from \eqref{tnhit1}, \eqref{tnhit2} and the fact that $0<\xi<1$ we see that
\begin{equation}\label{hit}
 n(1-\xi) = s +O(1) \hbox{\quad and\quad} \frac{d(1-\xi^r)}{r} n = m+O(d).
\end{equation}

We next verify that $n$ and $d$ satisfy the assumptions of Theorem~\ref{thLLT},
i.e., that $n\to\infty$, and $d$ and $\log n-d\to\infty$.
(The theorem assumes $d=d(n)$ is defined for every $n$, but there is no problem 
considering only a subsequence.)
Certainly, $n\ge \tn\ge s\to\infty$ as $s\to\infty$. We have already noted that $d=d(s)\to\infty$.
For the last condition, by \eqref{dsdef} and Lemma~\ref{asy} we have
\[
  d = \log(1/\xi) (1+O(\xi^{r-1})) = \bd + O(\bd e^{-\bd}).
\]
In particular,
\begin{equation}\label{dbdo1}
 d=\bd+o(1).
\end{equation}
Since $n=\tn+O(1) = s/(1-\xi)+O(1) \sim s$ we thus have
\begin{equation}\label{lnd}
 \log n -d = \log s-\bd +o(1) \to\infty.
\end{equation}
Thus all conditions of Theorem~\ref{thLLT} are satisfied. 

As in the statement of Theorem~\ref{thLLT}, set
\begin{equation}\label{psdef}
 p = d \frac{(r-1)!}{n^{r-1}}.
\end{equation}
In this section, $d$ and $n$ are functions of $s$, so this defines a function $p(s)$;
as usual, we suppress the dependence on $s$.
Let $\sigma_L^2 = n e^{-d}$ and $\sigma_M^2= nd/r$, as in \eqref{sLM}.
Under the assumptions of Theorem~\ref{thLLT}
(which we have just verified)
we have $\sigma_L\to\infty$ and $d=o(\sigma_M)$. Hence, by \eqref{hit} and
\eqref{ebds}, the values
$s$ and $m$ are within $o(1)$ standard deviations of the expectations of the numbers
$L_1$ and $M_1$ of vertices and edges in the largest component of the (binomial) random hypergraph $\Hrnp$.
Thus, by Theorem~\ref{thLLT},
\begin{equation}\label{p1}
 \Pr( L_1=s, M_1=m ) \sim \frac{1}{2\pi\sigma_L\sigma_M} \sim \frac{\sqrt{r}e^{d/2}}{2\pi n\sqrt{d}}.
\end{equation}

Recalling that $n\sim s$, for $s$ large enough we have $n<2s$, so the hypergraph $\Hrnp$ can
have at most one component with $s$ or more vertices. Writing $Z=Z(\Hrnp)$ for the number
of components with $s$ vertices and $m$ edges, we thus have
\[
 \Pr( L_1=s,M_1=m) = \Pr(Z=1) = \E[Z].
\]
Hence, by linearity of expectation,
\begin{equation}\label{p2}
  \Pr( L_1=s,M_1=m) = \binom{n}{s} C_r(s,m) p^m (1-p)^{M-m},
\end{equation}
where
\[
 M = \binom{n}{r} - \binom{n-s}{r}
\]
is the number of possible hyperedges meeting a given set of $s$ vertices.

From \eqref{p1} and \eqref{p2} we see that
\begin{equation}\label{Cform1}
 C_r(s,m) \sim \frac{\sqrt{r}e^{d/2}}{2\pi n\sqrt{d}} {\binom{n}{s}}^{-1} \bb{p^m (1-p)^{M-m}}^{-1}.
\end{equation}
The rest of the proof is `just' calculation, but this calculation is not so simple. A
significant hindrance is that we would eventually like to work in terms of $s$, $m$, $\bd=rm/s$
and the implicitly defined $\xi=\Phi_r^{-1}(\bd)$.
The quantity $\tn=s/(1-\xi)$ is a simple function of these variables, but $n=\ceil{\tn}$
is not.
Morally speaking, rounding to $n$ should make no difference, but showing this seems
to require some work: because of the large exponents appearing in \eqref{Cform1}, the
very small relative change of replacing $n$ by $\tn$ in the various factors in \eqref{Cform1}
can change these factors by a large amount even though, as we shall see, it does
not significantly change (a suitably adapted form of) the whole formula.
The last factor in \eqref{Cform1} is perhaps the hardest to deal with; fortunately, we can use a trick,
relating it to a binomial probability. The key point (established below) is that $Mp$ is rather
close to $m$.

Recall that, crudely,
\[
 n \sim \tn \sim s,
\]
and, from Lemma~\ref{asy} and \eqref{dbdo1}, that
\begin{equation}\label{xisim}
 \xi \sim  e^{-\bd} \sim e^{-d}.
\end{equation}
Thus,
\begin{equation}\label{xininf}
 \xi n \sim ne^{-\bd} \sim se^{-\bd} \to\infty,
\end{equation}
using \eqref{bds} in the last step. Turning to $n-s$, recalling
\eqref{tnhit1} and that $n=\ceil{\tn}$,
we have the rather accurate bound
\begin{equation}\label{nsacc}
 n-s = \tn-s +O(1) =\xi\tn +O(1) = \xi n +O(1).
\end{equation}
We shall need this later, though often the simpler consequence
\begin{equation}\label{nssimple}
 n-s \sim \xi n \sim n e^{-\bd}
\end{equation}
will suffice.

Since
\[
 r! \binom{x}{r} = x(x-1)\cdots (x-(r-1)) = x^r - \binom{r}{2} x^{r-1} + O(x^{r-2}),
\]
we have
\begin{eqnarray*}
 M &=& \frac{n^r-\binom{r}{2}n^{r-1} - (n-s)^r}{r!} + O( n^{r-2} + (n-s)^{r-1}) \\
 &=& \frac{n^r}{r!} \left(1-\binom{r}{2}\frac{1}{n} - \left(1-\frac{s}{n}\right)^r\right)
    +O(n^{r-2}+n^{r-1}e^{-(r-1)\bd}),
\end{eqnarray*}
using \eqref{nssimple} in the last step.
Since $r-1\ge 1$ and $n e^{-\bd}\to\infty$, we have
$n^{r-2}+n^{r-1}e^{-(r-1)\bd} = O(n^{r-1}e^{-\bd})$.
As the `cross term' $(\binom{r}{2}/n) (1-s/n)^r$ is of order
$O((n-s)^r n^{-r-1}) = O(n^{-1}e^{-\bd})$, we thus have
\begin{equation}\label{Macc}
 M = \frac{n^r}{r!} \left(1-\binom{r}{2}\frac{1}{n}\right)
 \left(1 - \left(1-\frac{s}{n}\right)^r\right) \left(1+O(n^{-1}e^{-\bd})\right).
\end{equation}

We shall need this accurate estimate later; for the moment, something simpler suffices.
From \eqref{nsacc} we have
\begin{equation}\label{ons}
 1-\frac{s}{n} = \frac{n-s}{n} = \frac{\xi n+O(1)}{n} = \xi+O(1/n).
\end{equation}
Hence \eqref{Macc} implies the cruder bound
\[
 M = \frac{n^r}{r!} (1 -\xi^r)(1+O(n^{-1})).
\]
Recalling the definition \eqref{psdef} of $p$, it follows that
\[
 Mp = \frac{dn}{r} ( 1-\xi^r) + O(d) = m+O(d),
\]
where the last step is from \eqref{hit}.
Now $p=o(1)$, while certainly $m=\bd s/r\to\infty$ and $d=o(\sqrt{m})$
(since $d\le\log n$ and $m/n\sim m/s\to\infty$). It follows that the probability
that a binomial random variable $\Bi(M,p)$ takes the value $m$ is asymptotically
\[
 \frac{1}{\sqrt{2\pi Mp(1-p)}} \sim \frac{1}{\sqrt{2\pi Mp}}\sim \frac{1}{\sqrt{2\pi m}}.
\]
In other words,
\[
 \binom{M}{m} p^m (1-p)^{M-m} \sim (2\pi m)^{-1/2}.
\]
Combining this with \eqref{Cform1} we see that
\begin{equation}\label{Cform2}
 C_r(s,m) \sim \frac{\sqrt{r}e^{d/2}}{2\pi n\sqrt{d}} {\binom{n}{s}}^{-1} \binom{M}{m} \sqrt{2\pi m}.
\end{equation}

From \eqref{dbdo1} we have $d\sim \bd$ and $e^{d/2}\sim e^{\bd/2}$.
Since $n\sim s$ and $rm/s=\bd$, we may simplify \eqref{Cform2} slightly to obtain
\begin{equation}\label{Cform3}
 C_r(s,m) \sim \frac{e^{\bd/2}}{\sqrt{2\pi s}} {\binom{n}{s}}^{-1} \binom{M}{m}.
\end{equation}
This formula may appear appealingly concise, but unfortunately it still involves
$n$, defined in a slightly unpleasant way (involving rounding), both directly
and in the definition of $M$. So we continue  with our manipulations.

Firstly, note for later that, from \eqref{ons},
\begin{equation}\label{sn}
 \frac{s}{n} = 1-\xi +O(1/n).
\end{equation}
As $s\to\infty$ we certainly have $n\to\infty$ and $n-s\to\infty$ (see \eqref{nssimple}
and \eqref{xininf}),
so by Stirling's formula and the estimate $s\sim n$ we have
\begin{equation*}
 {\binom{n}{s}}^{-1} = \frac{(n-s)!s!}{n!} \sim  \sqrt{2\pi(n-s)} (1-s/n)^{n-s} (s/n)^s.
\end{equation*}
Since $\frac{n-s}{s}\sim\frac{n-s}{n}\sim e^{-\bd}$ by \eqref{nssimple}, it follows from
\eqref{Cform3} that
\begin{equation}\label{Cformnew}
  C_r(s,m) \sim (1-s/n)^{n-s} (s/n)^s \binom{M}{m}.
\end{equation}

Now $M\sim n^r/r!$, while $m=\bd s/r\sim \bd n/r$. Since $\bd=O(\log n)$, for $r\ge 3$
it follows that $m^2=O(n^2\log^2 n)=o(M)$. Hence
\begin{equation}\label{Mm1}
  \binom{M}{m}\sim \frac{M^m}{m!}.
\end{equation}
For $r=2$ we have $m=\bd s/2$ and $M=n^2/2(1+O(1/n))=s^2/2(1+O(\xi))=s^2/2(1+O(e^{-\bd}))$,
using \eqref{sn} and \eqref{xisim}. Arguing as for \eqref{Nm2}, it follows that
\begin{equation}\label{Mm2}
 \binom{M}{m} \sim \frac{M^m}{m!} e^{-\bd^2/4}.
\end{equation}
Since, in any case, $m=o(n^2)$, we have
\[
 \left(1-\binom{r}{2}\frac{1}{n}\right)^m \sim e^{-\binom{r}{2}\frac{m}{n}}
 \sim e^{-\binom{r}{2}\frac{m}{s}} = e^{-(r-1)\bd/2},
\]
where in the second step we used that $s/n=1+O(e^{-\bd})=1+o(1/\bd)$ (from \eqref{sn})
and $m/n=O(\bd)$.
Also, since $m=O(\bd n)$, we have 
\[
 \left(1+O(n^{-1}e^{-\bd})\right)^m = \exp(O(\bd e^{-\bd})) \sim 1.
\]
Combining these estimates with \eqref{Mm1} and \eqref{Macc}, for $r\ge 3$ we find that
\begin{equation}\label{binMm}
 \binom{M}{m} \sim \frac{n^{rm}}{m!r!^m}  e^{-(r-1)\bd/2} \bb{ 1- (1-s/n)^r }^m .
\end{equation}
For $r=2$ we obtain the same formula with an extra factor of $e^{-\bd^2/4}$.
We write the formulae in the rest of the proof for the case $r\ge 3$; the remaining
estimates apply just as well when $r=2$, with the factor $e^{-\bd^2/4}$ inserted where appropriate.
From \eqref{Cformnew} and \eqref{binMm}, we see that
\[
  C_r(s,m) \sim e^{-(r-1)\bd/2} \frac{n^{rm}}{m!r!^m}  \bb{ 1- (1-s/n)^r }^m (1-s/n)^{n-s} (s/n)^s .
\]
We may rewrite this as
\[
 C_r(s,m) \sim e^{-(r-1)\bd/2}\frac{s^{rm}}{m!r!^m} (s/n)^{-rm}\bb{ 1- (1-s/n)^r }^m \bb{(1-s/n)^{n/s-1} (s/n)}^s  ,
\]
and hence as
\begin{equation}\label{Cformx}
 C_r(s,m) \sim e^{-(r-1)\bd/2}\frac{s^{rm}}{m!r!^m}
  \bb{ x^{-r}(1- (1-x)^r) }^m
  \bb{(1-x)^{1/x-1} x}^s,
\end{equation}
where $x=s/n$. We would like to replace $x$ by $s/\tn=1-\xi$.
First, we substitute $y=1-x$, obtaining
\[
 C_r(s,m) \sim e^{-(r-1)\bd/2}\frac{s^{rm}}{m!r!^m} g(y)^m h(y)^s,
\]
where
\[ 
 g(y) = (1-y)^{-r}(1-y^r) \hbox{\quad and\quad} h(y) = y^{y/(1-y)}(1-y).
\]

Set
\[ 
 f(y) = g(y)^m h(y)^s.
\]
It is straightforward to check that
\[
 \bb{ \log g(y) }' = r\frac{1-y^{r-1}}{(1-y^r)(1-y)} = r+O(y)
\]
as $y\to 0$, and
\[
 \bb{ \log h(y) }' = \frac{\log y}{(1-y)^2} = \log y + O(y|\log y|)  = \log y+o(1),
\]
so
\[
 \bb{ \log f(y) }' = rm + s\log y  +O(my)+o(s).
\]
We wish to compare $f(\xi)$ with $f(1-s/n)$. From \eqref{ons}
we have $1-s/n = \xi+O(1/n)$.
Hence we need only consider values of $y$ with $|y-\xi|=O(1/n)$.
In this range,
\[
 \log y = \log\xi + O(1/(n\xi)) = \log\xi +o(1),
\]
recalling that $n\xi\to\infty$.
By Lemma~\ref{asy}, $\log\xi = -\bd+o(1) = -rm/s+o(1)$. Hence
\[
 \bb{ \log f(y)}' = O(my)+o(s) = o(s),
\]
where in the final step we used that $my=O(\bd s\xi)=O(\bd s e^{-\bd})=o(s)$.
Since $1-s/n=\xi+O(1/n)=\xi+O(1/s)$ it follows that
\[
 f(1-s/n) \sim f(\xi).
\]
This is exactly what we need to allow us to replace $x=s/n$ by $x=1-\xi$ in \eqref{Cformx}.

In conclusion, writing $\rho=1-\xi$, for $r\ge 3$ we have
\begin{eqnarray*}
 C_r(s,m) &\sim& e^{-(r-1)\bd/2}\frac{s^{rm}}{m!r!^m}
  \bb{ \rho^{-r}(1- (1-\rho)^r) }^m
  \bb{(1-\rho)^{1/\rho-1} \rho}^s  \\
 &=& e^{-(r-1)\bd/2}\frac{s^{rm}}{m!r!^m} \exp(-sF_r(\bd)),
\end{eqnarray*}
where the last step is from \eqref{Frho}.

When $r=2$ we obtain the same formula with an extra factor of $e^{-\bd^2/4}$,
from using \eqref{Mm2} in place of \eqref{Mm1}. This proves \eqref{Cform}.

As noted above, \eqref{Pform} follows from \eqref{Cform}
by Lemma~\ref{lPC}, so the proof is complete.
\end{proof}


\section{Appendix}

In this appendix we briefly show that Theorem~\ref{thuniv} does indeed
extend the asymptotic formula given by Bender, Canfield and McKay~\cite{BCMcK}.
(It does not quite imply their result, since they have an explicit bound on the
$1+o(1)$ error term.)

Writing $P_2(s,m)$ for the probability that a random $m$-edge graph on $[s]$
is connected, Bender, Canfield and McKay showed that whenever $m=m(s)$
satisfies $m-s\to\infty$ and $m\le \binom{s}{2}-s$, then
\begin{equation}\label{bck}
 P_2(s,t) \sim e^{a(x)}\left(\frac{2e^{-x}y^{1-x}}{\sqrt{1-y^2}}\right)^s,
\end{equation}
where $x=m/s$, $y=y(x)$ is defined implicitly by
\begin{equation}\label{BCKdef}
 2xy=\log\left(\frac{1+y}{1-y}\right),
\end{equation}
and
\begin{equation}\label{adef}
 a(x) =x(x+1)(1-y)+\log(1-x+xy)-\tfrac{1}{2}\log(1-x+xy^2).
\end{equation}
Here we have changed the notation to match ours, and have simplified the more precise error term
given in~\cite{BCMcK}. Note that in our notation $x$ is simply $\bd/2$.

Let $\bd=2m/s$ and define $\xi$ as in \eqref{xidef} (with $r=2$), so
\begin{equation}\label{xidef2}
 \bd = \Phi_2(\xi) = \log(1/\xi)\frac{1+\xi}{1-\xi}.
\end{equation}
Set
\begin{equation}\label{yform}
 y  = \frac{1-\xi}{1+\xi}.
\end{equation}
Then from \eqref{xidef2} we have $\bd y = \log(1/\xi) = \log\left(\frac{1+y}{1-y}\right)$,
so \eqref{yform} defines the same $y=y(\bd)$ as in~\cite{BCMcK}.

Now using \eqref{yform}, $x=\bd/2$ and \eqref{xidef2} to write everything in terms
of $\xi$, one can check that
\begin{equation}\label{LF2}
 \log \left(\frac{2e^{-x}y^{1-x}}{\sqrt{1-y^2}}\right) 
 = -F_2(\bd),
\end{equation}
where $F_2(\bd)$ is defined in \eqref{Fdef}. (In fact, this computation is carried out in the
appendix to~\cite{smoothing}.)

Similarly, writing $G_2(\bd)$ (defined in \eqref{Gdef})
and $a(x)$ and hence $\exp(a(x))$ as a function of $\xi$,
using, for example, Maple, one can verify that 
\begin{equation}\label{aG2}
 \exp(a(x)) =  G_2(\bd).
\end{equation}
Indeed, it turns out that
\[
 x(x+1)(1-y) = \frac{ 2\bd\xi+\bd^2\xi }{2(1+\xi)} = g_2(\bd),
\]
\[
 1-x+xy = (1+\xi)^{-1} (1+\xi-\bd\xi) = (1+\xi)^{-1} a_2(\bd),
\]
and
\[
 1-x+xy^2 = (1+\xi)^{-2} \bb{ (1+\xi)^2-2\bd\xi } = (1+\xi)^{-2} b_2(\bd).
\]
These combine to give
\[
 \exp( a(x)) = \frac{a_2(\bd)}{\sqrt{b_2(\bd)}} e^{g_2(\bd)} = G_2(\bd)
\]
as claimed. By \eqref{LF2} and \eqref{aG2}
the $r=2$ case of the formula \eqref{universal2} in Theorem~\ref{thuniv}
does indeed match \eqref{bck}.

\end{document}